\documentclass[a4paper]{article}

\usepackage{amssymb,hyperref}

\usepackage{amsfonts,amsmath,amscd,amssymb,amsthm,color}

\usepackage{tikz-cd}
\usepackage[linesnumbered,ruled]{algorithm2e}
\usepackage{enumitem}
\usepackage{changes}

\usepackage{yfonts}

\definechangesauthor[color=blue]{S}
\definechangesauthor[color=green]{M}
\definechangesauthor[color=red]{R}

\usepackage[title]{appendix}

\newfont{\eufm}{eufm10 scaled\magstep1}

\newcommand{\cB}{\mathcal{B}}
\newcommand{\cC}{\mathcal{Z}}
\newcommand{\ccC}{\mathcal{C}}

\newcommand{\cE}{\mathcal{E}}

\newcommand{\cL}{\mathcal{L}}
\newcommand{\cF}{\mathcal{F}}

\newcommand{\cO}{\mathcal{O}}

\newcommand{\cS}{\mathcal{S}}

\newcommand{\cM}{\mathcal{M}}

\newcommand{\bbN}{\mathbb{N}}
\newcommand{\bbZ}{\mathbb{Z}}
\newcommand{\bbC}{\mathbb{C}}
\newcommand{\bbD}{\mathbb{D}}
\newcommand{\bbK}{\mathbb{K}}

\newcommand{\coC}{\mathbf{C}}
\def\Spec{{\rm Spec}}
\def\dres{\partial{\rm Res}}
\def\res{{\rm Res}}
\def\sdres{{\rm s}\partial{\rm Res}}
\def\ord{\rm ord}

\def\mod{\textrm{ mod }\,}
\def\Ker{\rm Ker}

\def\gcrd{\rm gcrd}

\def\mod{\rm mod}

\def\para{\vspace{1.5 mm}}

\def\gcd{\rm gcd}

\def\Im{\rm Im}
\def\BC{\texttt{BC}}
\def\e{\varepsilon_{P,Q}}
\def\veL{\varepsilon_{L}}
\def\ec{e_{P,Q}}
\def\eL{e_{L}}
\def\r{{\bf r}}
\def\s{{\rm S}}
\def\0{{\bf 0}}

\newtheorem{thm}{Theorem}[section]
\newtheorem{lem}[thm]{Lemma}

\newtheorem{cor}[thm]{Corollary}
\newtheorem{prop}[thm]{Proposition}
\newtheorem{defi}[thm]{Definition}
\newtheorem{rem}[thm]{Remark}

\begin{document}

\title{Computing defining ideals of space spectral curves\\ for algebro-geometric third order ODOs}


\author{Sonia L. Rueda\\
Dpto. de Matem\' atica Aplicada, E.T.S. Arquitectura\\
Universidad Polit\' ecnica de Madrid.\\
Avda. Juan de Herrera 4, 28040-Madrid, Spain.\\
\tt{sonialuisa.rueda@upm.es}
\and
Maria-Angeles Zurro\\
Dpto. de Matem\' aticas\\
Universidad Aut\' onoma de Madrid\\
Campus de Cantoblanco, Madrid, Spain\\
\tt{mangeles.zurro@uam.es}
}

\date{\today}
\maketitle

\thispagestyle{empty}

\begin{abstract}

Commuting pairs of ordinary differential operators (ODOs) have been related to plane algebraic curves since the work of Burchnall and Chaundy a century ago. We introduce now the concept of Burchnall-Chaundy (BC) ideal of a commuting pair, as the ideal of all constant coefficient bivariate polynomials satisfied by the pair. We prove this prime ideal to be equal to the radical of a differential elimination ideal and the defining ideal of a plane algebraic curve, the spectral curve of a commuting pair.

The ODOs of this work have coefficients in an arbitrary differential field with field of constants algebraically closed and of characteristic zero. 
Motivated by the extension of the recently introduced Picard-Vessiot theory for spectral problems {$L(y)=\lambda y$}, where $\lambda$ is an algebraic parameter, we also define the BC ideal of an algebro-geometric third order operator $L$. 
This allows a constructive proof of a famous theorem by I. Schur, establishing an isomorphism between the centralizer of $L$ and the coordinate ring of a space algebraic curve, that we define as the spectral curve of $L$ and whose defining ideal is the BC ideal of $L$. We provide the first explicit example of a non-planar spectral curve. We compute a set of generators of the defining ideal of this curve by means of differential resultants and define a new coefficient field determined by the spectral curve, to effectively compute an {intrinsic} right factor of $L-\lambda$.

\end{abstract}


\medskip

\textsc{\href{https://zbmath.org/classification/}{MSC[2020]}:}  13N10, 13P15, 12H05


\section{Introduction}\label{sec-intro}

A famous theorem by I. Schur \cite{Sch} implies that the centralizer of an ordinary differential operator, with analytic coefficients, has as quotient field a function field of one variable. Therefore such centralizers can be seen as affine rings of algebraic curves, \cite{PRZ2019}. In this work we will call these curves {\it spectral curves}. 

We consider ordinary differential operators with coefficients in an arbitrary differential field $\Sigma$, whose field of constants $\coC$ is algebraically closed and of zero characteristic. 
We use differential algebra, differential Galois theory and algebraic geometry to give a constructive proof of the following theorem.

\begin{thm}\label{thm-A}
    The centralizer of a third order ordinary differential operator $L$, with non constant coefficients in $\Sigma$, is either isomorphic to a polynomial ring in one variable or to the coordinate ring $\coC[\Gamma]$ of an irreducible (space) algebraic curve $\Gamma$.
\end{thm}
\noindent A differential operator $L$ whose centralizer is not a polynomial ring in one variable is often called {\it algebro-geometric} \cite{We}, \cite{Grig}. A fundamental goal driving this work is the development of a Picard-Vessiot theory for spectral problems $(L-\lambda)(y)=0$, in the case of an algebro-geometric differential operator $L$ and an algebraic parameter $\lambda$ over $\Sigma$. 

The authors defined spectral Picard-Vessiot fields for algebro-geometric Schrödinger operators in \cite{MRZ2} and provided algebraic tools {to effectively} compute a basis of solutions for  the spectral problem in this case. { The extension to higher order of the techniques developed for second-order operators, posts several important issues.  {The main problem treated in this paper is the definition and computation of spectral curves for algebro-geometric ODOs.} 

{
Computing algebro-geometric ODOs amounts to computing Gelfand-Dikii hierarchies and solutions of them \cite{GD}, which are non trivial problems. 
It is well known that solutions of the KdV hierarchy provide algebro-geometric Schrödinger ODOs \cite{MRZ1}, \cite{MRZ2}, \cite{Veselov}. Solutions of the Boussinesq hierarchy \cite{DGU} provide algebro-geometric operators of third order, and investigating those we discovered that planar algebraic curves do not explain all nontrivial centralizers of third order ODOs. We found an example of an algebro-geometric third order ODO whose centralizer is the commutative $\coC$-algebra generated by three ODOs, \cite{HRZ2023}. Explaining the nature of this example was the seed for the results presented here. A general study of algebro-geometric differential operators of prime order is postponed until meaningful examples are found, for order five and higher.} 

{We believe that the novel results that emerge from the study of the order three case will be important for the development of a general spectral Picard-Vessiot theory. 
Specifically, our main contribution is the identification of the spectral curve of $L$, as the curve defined by its centralizer, and not by a pair of commuting ODOs. As an application, we explain how the field of functions on the spectral curve of $L$ is essential to obtain an intrinsic right factor of $L-\lambda$. 
To the best of our knowledge, this fact is established for the first time in this work.}

\medskip

Planar spectral curves have been commonly defined  for pairs of commuting differential operators $P$ and $Q$, since the seminal work of Burchnall and Chaundy \cite{BC1}, see also \cite{Mum2}, \cite{K78} and the recent review \cite{Zheglov}. Most results have been obtained in the case of pairs $P,Q$ with coprime orders, called the rank $1$ case, see for instance \cite{Zheglov}. Without this restriction, hence for any rank $r = \gcd ((\ord (P) , \ord (Q))$, we define an isomorphism between the coordinate ring of the planar spectral curve (an irreducible plane algebraic curve) $\Gamma_{P,Q}$ and the commutative ring of differential operators $\coC[P,Q]$. Namely, we prove that the ideal of the curve  $\Gamma_{P,Q}$ equals the ideal of all constant coefficient polynomials $g(\lambda,\mu)$ that are satisfied by the differential operators, commonly denoted $g(P,Q)=0$ and known as {\it Burchnall and Chaundy polynomials}. We introduce in this paper the notion of  Burchnall and Chaundy ideal of a pair $P,Q$ and denote it by $\BC(P,Q)$. Since $P$ and $Q$ are differential operators in an euclidean domain the ideal $\BC(P,Q)$ is naturally proved to be a prime ideal. These results are included in Theorem \ref{thm-CPQ}. From now on we will call $\Gamma_{P,Q}$ the {\it spectral curve of the pair}  $P,Q$ to distinguish it from the spectral curve associated with the centralizer of an ordinary differential operator.

It was known by E. Previato \cite{Prev} and G. Wilson \cite{W1985}, in the case of differential operators with analytic coefficients, that a Burchnall and Chaundy polynomial for a commuting pair $P,Q$ can be obtained computing the differential resultant $\dres(P-\lambda,Q-\mu)$ of $P-\lambda$ and $Q-\mu$. In particular, they proved that such resultant is a constant coefficient polynomial. 
We generalize this result to differential operators with coefficients in $\Sigma$ in Theorem \ref{thm-Prev}. But we go further proving that the ideal of all Burchnall and Chaundy polynomials for the pair $P,Q$ is generated by an irreducible polynomial $f(\lambda,\mu)$, the radical of $\dres(P-\lambda,Q-\mu)$, see Theorem \ref{thm-BCf}. For this purpose, we use the differential elimination ideal 
determined by $P-\lambda$ and $Q-\mu$,
\[\cE(P-\lambda, Q-\mu)=(P-\lambda, Q-\mu)\cap \Sigma[\lambda,\mu],\]
the ideal of all differential operators generated by $P-\lambda$ and $Q-\mu$ in $\Sigma[\lambda,\mu]$, for which the derivation has been eliminated. Our main result about arbitrary commuting pairs of ODOs establishes the intrinsic nature of the  differential resultant of $P-\lambda$ and $Q-\mu$. We prove the following result in Theorem \ref{thm-BCPQ}.
\begin{thm}\label{thm-B}
    Given ordinary differential operators $P$ and $Q$ with coefficients in $\Sigma$, let
    $f(\lambda,\mu)$ be the radical of $\dres(P-\lambda,Q-\mu)$.
    The following statements hold:
    \begin{enumerate}
        \item The ideal $(f)$ generated by $f$ in $\coC[\lambda,\mu]$ equals
 $\BC(P,Q)$.

        \item The differential ideal $[f]$ generated by $f$ in $\Sigma[\lambda,\mu]$ equals the radical of the elimination ideal $\cE(P-\lambda, Q-\mu)$. 
    \end{enumerate}
\end{thm}

 For an algebraically closed field of constants, the classical theory of Picard and Vessiot \cite{VPS} allows to formally manage solutions. It will be important  to give formal proofs of our results. Specializing the parameter $\lambda$ to $\lambda_0\in \coC$, the  Picard-Vessiot extension of $\Sigma$ for $(P-\lambda_0)(y)=0$ is the splitting field for $(P-\lambda_0)(y)=0$. If algebraic variables $\lambda$ and $\mu$ over $\Sigma$ are considered, the differential operators $P-\lambda$ and  $Q-\mu$ belong to the differential field $\cF=\Sigma(\lambda,\mu)$ and  to use classical Picard-Vessiot theory, the algebraic closure of $\cF$ can be considered, whose field of constants is algebraically closed.

\medskip

Centralizers are maximal commutative rings, \cite{PRZ2019}.
In some cases, the centralizer $\cC(L)$ of a differential operator $L$ is the affine ring of a planar curve, as in the case of algebro-geometric Schrödinger operators where $\cC(L)=\coC[L,A]$. 
Let us fix now an algebro-geometric differential operator $L$ of third order with coefficients in $\Sigma$.
One of the motivations to give a constructive proof of Theorem \ref{thm-A} is to show that in general $\cC(L)$ is not the coordinate ring of planar curve. 
As a consequence of results of K. Goodearl in \cite{Good}, it has the structure of a free $\coC[L]$-module 
\[\cC(L)=\coC[L,A_1,A_2]=\coC[L]\oplus \coC[L]A_1\oplus \coC[L]A_2,\]
where $A_i$ is of minimal order congruent with $i\ (\mod\ 3)$. 
Observe that $\cC(L) $ contains the rank one algebras $\coC[L,A_i] $, but also $\coC[A_1 ,A_2 ]$ whose rank may be grater than one, {being the rank of a set the greatest common divisor of the orders of its elements, see \cite{PRZ2019}.}  We introduce the notion of Burchnall-Chaundy ideal $\BC(L)$ of $L$ to explain the ring multiplicative
 structure of $\cC(L)$. This is the ideal of all polynomials $g(\lambda,\mu_1,\mu_2)$ in $\coC[\lambda,\mu_1,\mu_2]$ such that $g(L,A_1,A_2)=0$. 
We prove that $\BC(L)$ is a prime ideal that naturally contains the Burchnall-Chaundy ideals of pairs of differential operators in the centralizer of $L$. Three BC ideals are of crucial importance in this work
\[\BC(L,A_1)=(f_1),\ \BC(L,A_2)=(f_2),\ \BC(A_1,A_2)=(f_3),\]
all of them contained in $\BC(L)$, but even more so, we conclude that $\BC (L)= (f_1 , f_2 , f_3 )$, see Corollary \ref{cor-Gamma}. The constructive proof of Theorem \ref{thm-A} is achieved in Corollary \ref{cor-Gamma}, where $\cC (L)$ is shown to be isomorphic to the coordinate ring 
\[\coC[\Gamma]=\frac{\coC[\lambda,\mu_1,\mu_2]}{\BC(L)}=\frac{\coC[\lambda,\mu_1,\mu_2]}{(f_1,f_2,f_3)}\]
of an irreducible space algebraic curve $\Gamma$ in $\coC^3$, {that we define as} the {\it spectral curve of $L$}.

\medskip

\medskip

An important consequence of the constructive proof of Theorem \ref{thm-A} is to discover the appropriate coefficient field where a third order algebro-geometric operator $L-\lambda$ would have an intrinsic right factor. 
Recall that $\lambda$ is not a free parameter, it is governed by the space spectral curve $\Gamma$ of $L$.  Extending the ideal $\BC(L)$ to the ring of polynomials in three variables with coefficients in $\Sigma$, we can think of an extended curve  whose ring of regular functions is 
\[
\Sigma[\Gamma]=\frac{\Sigma[\lambda,\mu_1,\mu_2]}{[\BC(L)]}.
\]
Over its quotient field $\Sigma(\Gamma)$, a factorization of $L-\lambda$ is guaranteed,
\[
L-\lambda = N\cdot (\partial+\phi),
\]
where $\partial+\phi$ can be computed by means of the greatest common right divisors
$\gcrd(L-\lambda,A_1-\mu_1)$ or $\gcrd(L-\lambda,A_2-\mu_2)$, which are proved to coincide over $\Sigma(\Gamma)$.
The precise statement is contained in Theorem \ref{thm-factor}. We can think of $\partial+\phi$ as a global factor,  since  for almost every  point $P_0 \in \Gamma$ our methods produce a factorization of   $L-\lambda_0 = N_{P_0 } \cdot (\partial + \phi (P_0 ))$. That is why we call $\partial +\phi$ the intrinsic right factor of the  operator $L-\lambda$.

\bigskip

It remains as a future project to define and prove existence of the spectral Picard-Vessiot field of $L-\lambda$, for an algebro-geometric differential operator $L$. This will be the generalization of the spectral Picard-Vessiot fields for Schrödinger operators introduced in \cite{MRZ2}. It would be a differential field extension of $\Sigma(\Gamma)$, the minimal extension containing all the solutions, and it requires a full factorization of $L-\lambda$ over $\Sigma(\Gamma)$. This requires further investigations where the parametric Picard-Vessiot theory, introduced by Cassidy and Singer in \cite{CS}, and studied in \cite{Arr} and \cite{MO2018} may be relevant. 

\medskip

{\it The paper is organized as follows}. In Section \ref{sec-centralizers} we establish the appropriate framework to study  algebro-geometric differential operators, in the case of a third order operator $L$, whose definition is established in Corollary \ref{cor-AG}. The 
novel notion of  Burchnall-Chaundy (BC) ideal, of a pair $P, Q$ of ODOs and the  BC ideal $\BC(L)$ of a fixed differential operator $L$, are introduced in Section \ref{sec-BC}. These BC ideals are proved to be prime ideals describing the algebra $\coC [P, Q]$ and the centralizer $\cC(L)$, as coordinate rings of algebraic curves, respectively in theorems \ref{thm-CPQ} and \ref{thm-CL}. These allow the definitions of the spectral curves of a pair $P,Q$ and of a third order operator $L$. 

Sections  \ref{sec-generators} and \ref{sec-cenCoord} are dedicated to compute generators of these BC ideals.  Section \ref{sec-generators} contains the proof of Theorem \ref{thm-B}, the first part in Theorem \ref{thm-BCf} and the second in Theorem \ref{thm-BCPQ}. 
The effective description of $\BC(L)$ is achieved by means of any basis of the centralizer $\cC(L)$ and the computation of differential resultants. The choice of a basis of $\cC(L)$ is discussed in Section \ref{sec-normalizedbasis}. We use Gr\" obner bases in Section \ref{sec-genBC} to control the projection of the space spectral curve onto the plane $\lambda=0$ and obtain a set of generators of $\BC(L)$. 

In Section \ref{sec-factoring-Bsq}, we show that the algebro-geometric hypothesis implies the existence of a right factor $\partial+\phi$ of $L-\lambda$, over a new differential field of coefficients $\Sigma(\Gamma)$, defined in Section \ref{sec-DiffField}. The computation of $\phi$ is delicate and it is achieved by means of differential subresultants in Section \ref{sec-intrinsicRFactor}.
Finally, Section \ref{sec-Examples} contains the first computed example of a non planar spectral curve. It is used {to illustrate how} to globally factor the linear operator $L-\lambda$. All computations in this work are made with Maple \cite{Maple}.


To make this work as self contained as possible, two appendices are included. Appendix \ref{app-Centralizers} contains the results needed to understand centralizers of differential operators, giving special importance to Goodearl's construction of a basis of $\cC(L)$. In Appendix \ref{sec-DiffRes}, we review definitions and proofs of results on differential resultants and subresultants, in relation with the factorization of ODOs.  

\medskip


\noindent {\bf Notation.} For concepts in differential algebra we refer to \cite{VPS}, \cite{CH},  or \cite{Morales}.
A differential ring is a ring $R$ with a derivation $\partial$ on $R$. A differential ideal $I$ is an ideal of $R$ invariant under the derivation. 
We denote by
\begin{equation*}
    Const(R) = \{ r\in R \ | \ \partial (r) =0\} \ ,
\end{equation*}
which is called the ring of constants of $R$. 
Assuming that $R$ is a differential domain, its field of fractions $Fr (R)$ is a differential field with extended derivation 
$$\partial (f/g) = ( \partial (f)g-f\partial (g))/g^2.$$
A differential field $(\Sigma, \partial )$ is a differential ring which is a field. Given $a\in \Sigma$ we denote $\partial(a)$ by $a'$. Note that $Const(\Sigma)$ is a field whenever $\Sigma$ is. We assume that $C:=Const(\Sigma)$ is algebraically closed and has characteristic $0$.

\medskip

Let us consider algebraic variables $\lambda$ and $\mu$ with respect to $\partial$. Thus $\partial(\lambda) = 0$ and $\partial(\mu) = 0$ and we can extend the derivation $\partial$ of $\Sigma$ to the polynomial ring $\Sigma[\lambda, \mu]$, {having ring of constants $\coC[\lambda,\mu]$}.  {Similarly, we consider algebraic variables  $\mu_1$ and $\mu_2$ with respect to $\partial$ and extend the derivation to the polynomial ring $\Sigma[\lambda, \mu_1,\mu_2]$, {having ring of constants $\coC[\lambda,\mu_1,\mu_2]$}.}

\medskip

\noindent We denote by $\mathbb{N}$ the set of positive integers including $0$. {Notation regarding ODOs and their centralizers is included in Appendix  \ref{app-Centralizers}.} {The structure of the ring of differential operators $\Sigma[\partial]$ is reviewed in Appendix \ref{sec-DiffRes}.}

\section{Algebro-geometric   ODOs}\label{sec-centralizers}

Let us consider a third order ordinary differential operator $L$ in $\Sigma[\partial]$, with non constant coefficients, that is
$L\in \Sigma[\partial]\backslash \coC[\partial]$. We will assume throughout this work, that $L$  has a non trivial centralizer, as in Definition \ref{def-nontrivial}, this means that the centralizer of $L$ in $\Sigma [\partial]$
\begin{equation*}
    \cC(L)=\{A\in \Sigma [\partial]\mid LA=AL\},
\end{equation*}
does not equal the polynomial ring $\coC[L]$.
Therefore we have a chain of strict inclusions
\[\coC[L]\subset \cC(L)\subset \Sigma [\partial].\]
In Appendix \ref{app-Centralizers}, we review results of K. Goodearl in \cite{Good} regarding the structure of centralizers. 

\medskip

By Theorem \ref{thm-Appcen},1 the centralizer $\cC(L)$ is a free $\coC [L]$-module of rank $3$. Let $\cB(L)=\{1,A_1,A_2\}$ be a basis of $\cC(L)$ as a $\coC [L]$-module. The construction of a basis is reviewed in Appendix \ref{thm-Appcen}. Each $A_i$ is a monic operator in $\cC(L)\backslash \coC[L]$ of minimal order $o_i:=\ord(A_i)\equiv i (\mod \ 3 )$. Therefore
\begin{equation}\label{eq-cen_mod}
\cC(L)    =\coC[L]\oplus \coC[L]A_1\oplus \coC[L]A_2 =\{p_0(L)+p_1(L)A_{1}+p_2(L)A_{2}\mid p_i \in \coC[\lambda]\}.  
\end{equation}
From the composition of the basis we obtain the following result.

\begin{cor}\label{cor-AG}
Let $L$ be a third order operator in $\Sigma [\partial]\backslash \coC[\partial]$. The following statements are equiva\-lent:
\begin{enumerate}
    \item $L$ has a nontrivial centralizer $\cC(L)\neq \coC[L]$.
    
    \item There exists an operator $A$ in $\Sigma[\partial]$ of order $m$, relatively prime with $3$, such that $[L,A]=0$.
\end{enumerate}
\end{cor}

Differential operators satisfying the second statement of Corollary \ref{cor-AG} are usually called {\sf algebro-geometric operators}, see for instance \cite{We} or \cite{Grig}. This statement   highlights that $\cC(L)$ contains a pair of operators $L, A$ of rank one, being the rank of the pair the greatest common divisor of their orders. A discussion on the rank of a set of differential operators appears in \cite{PRZ2019}. 

\medskip

{In addition, $\cC(L)$ is also proved to be a commutative domain, see Appendix \ref{app-Centralizers}}. In fact cen\-tralizers $\cC(L)$ are maximal commutative subrings of $\Sigma [\partial]$, see \cite{PRZ2019}.
So, given a differential operator $M$ that commutes with $L$, we have the sequence of inclusions 
\begin{equation*}
\coC [L]\subseteq \coC [L,M] \subseteq \cC(L) ,  
\end{equation*} 
and all of them could be strict. In the case of a third order operator the following ring diagram of inclusions illustrates the ring structure of $\cC (L)$.
\begin{equation}\label{diagram}
\begin{tikzcd}
                                         &  & {\coC[L ,A_{1}]} \arrow[rrd]   &  &                  \\
{\coC[L ]} \arrow[rru]  \arrow[rrd] &  & {\coC[A_{1} ,A_{2}]} \arrow[rr] &  & {\coC[L ,A_{1},A_{2} ]}=\cC(L) \\
                                          &  & {\coC[L ,A_{2}]} \arrow[rru]   &  &   
\end{tikzcd}
\end{equation}

\begin{rem}\label{rem-BCpair}
    Making use of the ring structure of $\cC(L)$, it may happen that $A_1$ equals a polynomial in $A_2$ or vice versa, $A_1=q(A_2)$, with  $q$ a univariate polynomial. In this situation $\cC(L)=\coC [L,A_2]$. For instance, this is the case whenever $\ord (A_2 )=2$, then $A_1=A_2^2$ has order $4$ and $\cC(L)=\cC(A_2)$ is the centralizer of a second order operator.
\end{rem}

Let us denote by $\Sigma((\partial^{-1}))$ the ring of pseudo-differential operators in $\partial$ with coefficients in $\Sigma$, defined as \cite{Good}
\[\Sigma((\partial^{-1}))=\left\{\sum_{i=-\infty}^n a_i\partial^i\mid a_i\in \Sigma, n\in\bbZ\right\},\]
where $\partial^{-1}$ is the inverse of $\partial$ in $\Sigma((\partial^{-1}))$, $\partial^{-1}\partial=\partial\partial^{-1}=1$. Observe that 
\[\cC(L)\subset \Sigma[\partial]\subset \Sigma((\partial^{-1})).\]

The generalization of  Schur's theorem in \cite{Sch}, to differential operators with coefficients in a differential field $\Sigma$ has a long history, see for instance \cite{Good}, and the references given in Sections 3 and 4.
By \cite{Good}, Theorem 3.1 any monic third order operator $L$ has a unique cubic root up to multiplication by a cubic root of unity, denoted $L^{1/3}$, that determines the centralizer in $\Sigma((\partial^{-1}))$ of $L$,  denoted by $\cC((L))$ and equal to
\[
\cC((L))=\left\{ \sum_{j=-\infty}^N c_j (L^{1/3})^j\mid c_j\in {\bf C}\right\}.
\]
This implies that $\cC((L))$ is commutative and therefore
\[
\cC(L)=\cC((L))\cap \Sigma [\partial]
\]
is also a commutative differential domain,  see \ref{coro-domain} in the Appendix.  Moreover, as observed in \cite{PRZ2019}, Section 2.1 the centralizer $\cC(L) $ is a maximal commutative subring of $\Sigma[\partial]$  and, in a formal sense, $\Spec ( \cC(L))$ is an (abstract) algebraic curve since the transcendence degree of $\cC((L))$ over $\coC$ is $1$. In this work, we identify a defining ideal for this abstract curve. This is Corollary \ref{cor-Gamma} obtained as a consequence of Theorem \ref{thm-BCL}.

\section{Burchnall\textendash Chaundy ideals and spectral curves}\label{sec-BC}

A polynomial $f(\lambda,\mu)$ with constant coefficients satisfied by a commuting pair of differential ope\-rators $P$ and $Q$ is called a {\sf Burchnall-Chaundy  (BC) polynomial of the pair} $P$, $Q$, since the first result of this sort appeared is the $1923$ paper \cite{BC1} by Burchnall and Chaundy.  
We define in this section the ideal of all BC-polynomials associated to an arbitrary pair of commuting differential operators. For a fixed differential operator $L$ of order $3$, we also define an ideal of BC-polynomials. 

\subsection{Burchnall\textendash Chaundy ideal of a pair}

Given commuting differential operators $P$ and $Q$ in $\Sigma[\partial]$, {to avoid meaningless situations we will assume that $P,Q\notin \coC[\partial]$ and both have positive order.}
We define the ring homomorphism
\begin{equation}\label{def-ePQ}
    \ec: \coC [\lambda,\mu] \rightarrow \Sigma[\partial] \mbox{ by } \ec(\lambda)=P,\,\,\,  \ec(\mu)=Q.
\end{equation}
It is a ring homomorphism since $P$ and $Q$ commute, and they commute with the elements of $\coC$, the field of constants of $\Sigma$.
The image of $\ec$ is the $\coC$-algebra 
\[\coC[P,Q]=\left\{\sum_{i,j}  \sigma_{i,j} P^i Q^j \ \mid \ \sigma_{i,j}\in \coC \right\}.\] 
Observe that $\Ker(\ec)$ is an ideal of $\coC [\lambda,\mu]$.
In this setting, given $g\in \coC [\lambda,\mu]$ we will denote
\begin{equation}
    g(P,Q):=\ec (g).
\end{equation}
Thus a polynomial $g$ belongs to the kernel of $\ec$ if and only if $g(P,Q)=0$, that is 
\[\Ker(\ec)=\{g\in\coC [\lambda,\mu]\mid g(P,Q)=0\}.\]

\begin{defi}
Given commuting differential operators $P$ and $Q$ in {$\Sigma [\partial]\backslash \coC[\partial]$, both of positive order}, we define the {\sf BC-ideal of a pair} $P$ and $Q$ as 
\[\BC (P,Q):=\Ker(\ec) =\{g\in\coC [\lambda,\mu]\mid g(P,Q)=0\}.\]
We will call {\sf BC-polynomials} the elements of the BC-ideal.
\end{defi}

An important consequence of working with differential operators in an euclidean domain $\Sigma[\partial]$ is the following lemma.

\begin{lem}\label{lem-prime}
Given commuting differential operators $P$ and $Q$ in {$\Sigma [\partial]\backslash \coC[\partial]$, both of positive order},  the ideal $\BC(P,Q)$ is a prime ideal in $\coC [\lambda,\mu]$. 
\end{lem}
\begin{proof}
To start $\BC (P,Q)$ is a nonzero ideal by \cite{Good}, Theorem 1.13. For details, let us consider the centralizer $\cC(P)$ of $P$. Since $\cC(P)$ is a finitely generated $\coC[P]$-module there exists $p_0,\ldots ,p_t\in \coC[\lambda]$ such that
\[g(\lambda,\mu)=p_0(\lambda)+p_1(\lambda)\mu+\cdots +p_{t-1}(\lambda)\mu^{t-1}+\mu^t\in \BC(P,Q).\]

Given $h_1$ and $h_2$ in $\coC[\lambda,\mu]$, let us assume that $h_1\cdot h_2\in \BC(P,Q)$. Observe that $h_1(P,Q)=\ec(h_1)$ and $h_2(P,Q)=\ec(h_2)$ are differential operators in the euclidean domain $\Sigma[\partial]$. Since $\ec$ is a ring homomorphism \[0=\ec(h_1\cdot h_2)=\ec(h_1)\cdot \ec(h_2)\]
implies $\ec(h_1)=0$ or $\ec(h_2)=0$. Thus $h_1\in \BC(P,Q)$ or $h_2\in \BC(P,Q)$, proving that $\BC(P,Q)$ is prime.  
\end{proof}

\begin{defi}
Given commuting differential operators $P$ and $Q$ in {$\Sigma [\partial]\backslash \coC[\partial]$, both of positive order}, we define the {\sf spectral curve of the pair $P, Q$} as the irreducible algebraic variety of the prime ideal $\BC (P,Q)$, that is 
\begin{equation}\label{eq-GammaPQ}
  \Gamma_{P,Q}:=V(\BC(P,Q))=\left\{ \ (\lambda_0,\mu_0)\in\coC^2\mid g(\lambda_0,\mu_0)=0, \ \textrm{for all }  \ g\in\BC(P,Q) \ \right\}.  
\end{equation}
\end{defi}

\begin{rem}\label{rem-GammaPQ}
    We proved in Lemma \ref{lem-prime} that $\BC (P,Q)$ is a prime ideal. Then the algebraic variety $\Gamma_{P,Q} $ is an irreducible algebraic variety in $\coC^2$.  This variety cannot be a single point $(\lambda_0,\mu_0)$, because this would mean that $\BC(P,Q)=(\lambda-\lambda_0,\mu-\mu_0)$ implying $P-\lambda_0=0$ and $Q-\mu_0=0$, which contradicts $P,Q\notin \coC$. Therefore $\Gamma_{P,Q}$ is an irreducible algebraic curve in $\coC^2$.
\end{rem}

We obtain the following result as a consequence of the  construction of $\ec$ and Lemma \ref{lem-prime}.

\begin{thm}\label{thm-CPQ}
Given differential operators $P$ and $Q$ in {$\Sigma [\partial]\backslash \coC[\partial]$, both of positive order}, if $P$ and $Q$ commute then $\coC[P,Q]$   is a domain isomorphic to $\coC[\lambda,\mu]/\BC(P,Q)$, the coordinate ring of the spectral curve $\Gamma_{P,Q}$.
\end{thm}

\subsection{Burchnall\textendash Chaundy ideal of a third order operator}

Let us consider a third order operator $L$ in $\Sigma [\partial]\backslash \coC [\partial]$ with nontrivial centralizer $\cC (L)$.  We will define next the Burchnall\textendash Chaundy ideal of the operator $L$.

Given a basis $\cB (L)=\{1,A_1,A_2\}$ of $\cC(L)$,
we can define the ring homomorphism
\begin{equation}\label{eq-eL}
    \eL: \coC [\lambda,\mu_1, \mu_2] \rightarrow \Sigma[\partial] \mbox{ by } \ec(\lambda)=L,\,\,\,  \ec(\mu_1)=A_1,\,\,\, \ec(\mu_2)=A_2.
\end{equation}
We will denote monomials in $\coC [\lambda,\mu_1, \mu_2]$ by $M_{\alpha}=\lambda^{\alpha_0}\mu_1^{\alpha_1}\mu_2^{\alpha_2}$, $\alpha=(\alpha_0,\alpha_1,\alpha_2)\in\bbN^3$, whose image is $\eL(M_{\alpha})=L^{\alpha_0}A_1^{\alpha_1}A_2^{\alpha_2}$.
Thus the image of $\eL$ is the centralizer of $L$, 
\[\cC(L)=\coC[L,A_1,A_2]=\left\{\sum_{\alpha\in\bbN^3}  \sigma_{\alpha} \eL(M_{\alpha}) \ \mid \ \sigma_{\alpha}\in \coC \right\}.\]

Given $g\in \coC [\lambda,\mu_1, \mu_2]$ we will denote
\begin{equation}
    g(L,A_1,A_2):=\eL (g).
\end{equation}

\begin{defi}
Given a third order operator $L$ in $\Sigma[\partial]\backslash \coC [\partial]$, with nontrivial centralizer  and given a basis $\cB (L)=\{1,A_1,A_2\}$ of $\cC(L)$, we define a {\sf BC-ideal of an operator} $L$ as 
\begin{equation}\label{def-BCL}
\BC (L):=\Ker(\eL)=\{g\in\coC [\lambda,\mu_1, \mu_2]\mid g(L,A_1,A_2)=0\}.    
\end{equation}
We will call the elements of the BC ideal  {\sf BC-polynomials}.
\end{defi}

\begin{lem}\label{lem-prime3}
Given a third order operator $L$ in $\Sigma[\partial]\backslash \coC [\partial]$, with nontrivial centralizer $\cC (L)$,  the ideal $\BC(L)$ is a prime ideal in $\coC [\lambda,\mu_1, \mu_2]$. 
\end{lem}
\begin{proof}
Given a basis $\cB (L)=\{1,A_1,A_2\}$. Since $L$ and $A_i$ commute, by Lemma \ref{lem-prime} $\BC(L,A_i)$ is a nonzero ideal. Observe that $\BC(L,A_i)\subseteq \BC(L)$ which implies that $\BC(L)$ is also a nonzero ideal.

As in Lemma \ref{lem-prime} we can prove that $\BC(L)$ is a prime ideal using that $\eL$ is a ring homomorphism and that $\Sigma[\partial]$ is an euclidean domain.
\end{proof}

We obtain the following result as a consequence of the  construction of $\e$ and Lemma \ref{lem-prime3}.

\begin{thm}\label{thm-CL}
Given a third order operator $L$ in $\Sigma[\partial]\backslash \coC [\partial]$, with nontrivial centralizer $\cC (L)$, and given any basis $\{1,A_1,A_2\}$ of $\cC(L)$  then  the next isomorphism holds
\[\cC (L)=\coC[L,A_1,A_2]\simeq\frac{\coC [\lambda,\mu_1, \mu_2]}{\BC(L)}.\]
\end{thm}

It is well known, at least in the case of  analytic coefficients, that $\cC(L)$ is the affine ring of an algebraic curve.

\begin{defi}
Given a third order operator $L$ in $\Sigma[\partial]$, with nontrivial centralizer $\cC (L)$, we define the {\sf spectral curve of 
$L$} as the irreducible algebraic variety of the prime ideal $\BC (L)$, that is 
\[\Gamma_{L}:=\{P_0\in\coC^3\mid g(P_0)=0, \forall g\in\BC(L)\}.\]
\end{defi}

We will  guarantee in Section \ref{sec-genBC} that $\Gamma_L$ is in fact a curve. In general, the spectral curve of a third order operator will be a space curve and only in some special cases it will be a planar curve. 


\begin{rem}
    In some cases, described in Remark \ref{rem-BCpair}, if $\cC(L)=\coC [L,A]$ then $\BC(L)=\BC(L,A)$. Therefore the spectral curve of $L$ would be the spectral curve of the pair $L, A$, a plane curve. 
\end{rem}

\section{Elimination ideals for commuting pairs of ODOs}\label{sec-generators}

Given a pair of commuting differential operators $P, Q$ in {$\Sigma [\partial]\backslash \coC[\partial]$, both of positive order}, we know that the BC-ideal $\BC (P,Q)$ is a prime ideal by Lemma \ref{lem-prime}. This section is dedicated to the proof of Theorem \ref{thm-B}, that is obtained as a consequence of Theorem \ref{thm-BCPQ} in this section. Moreover, we start proving in Theorem \ref{thm-Prev}, a generalization   of a known result of E. Previato, \cite{Prev} and G. Wilson \cite{W1985} to the case of a pair of commuting operators with coefficients in $\Sigma$.

\subsection{Generalized Previato\textendash Wilson's Theorem}

Let us assume that $P$ and $Q$ have positive orders $n$ and $m$, and leading coefficients $a_n$ and $b_m$ respectively. 
Observe that the operators $P-\lambda$ and $Q-\mu$ have coefficients in the differential domain $(\Sigma[\lambda,\mu],\partial)$, see the notation in Section \ref{sec-intro}.
Let us consider the differential resultant of $P-\lambda$ and $Q-\mu$, as defined in Appendix \ref{sec-DiffRes}, 
\begin{equation*}
h(\lambda,\mu):=\dres(P-\lambda,Q-\mu).    
\end{equation*}
By \eqref{eq-S0}, it provides a nonzero polynomial in $ \Sigma [\lambda ,\mu]$, 
\begin{equation*}
 \label{eq-drespol}
    h(\lambda,\mu)= a_n^m\mu^n-b_m^n\lambda^m+\cdots
    =a_n^m\mu^n+\sum_{j=0}^{n-1} a_j(\lambda) \mu^j\neq 0,   
\end{equation*}
where $a_j(\lambda)$ belong to $\Sigma[\lambda]$ and have degree less than or equal to $m$.

\para 

The next theorem was proved by E. Previato in \cite{Prev}, for differential operators whose coefficients are analytic functions, see also \cite{W1985}. For completion, we review the proof given in \cite{MRZ1}, in this occasion for differential operators with coefficients in an arbitrary differential field $\Sigma$ with algebraically closed field of constant $\coC$ of zero characteristic.

\begin{thm}\label{thm-Prev}
    Let us consider differential operators $P$ and $Q$ in {$\Sigma [\partial]\backslash \coC[\partial]$, both of positive order}. If $P$ and $Q$ commute then $\dres(P-\lambda,Q-\mu)$ is a polynomial in $\coC[\lambda,\mu]$.
\end{thm}
\begin{proof}
   Recall that $P-\lambda$ and $Q-\mu$ are differential operators with coefficients in $\Sigma[\lambda,\mu]$, whose field of fractions is $\cF=\Sigma(\lambda,\mu)$. We can extend the derivation of $\cF$ to its algebraic closure $\overline{\cF}$, whose field of constants $\ccC$ is known to be algebraically closed, see \cite{Bron}, Corollary 3.3.1. Since the ring of contstants of $\Sigma[\lambda,\mu]$ is $\coC [\lambda,\mu]$ then it holds that $\Sigma[\lambda,\mu]\cap \ccC=\coC [\lambda,\mu]$.
    
    Let us consider a fundamental systems of solutions $\psi_1,\ldots ,\psi_n$ of $(P-\lambda)(y)=0$ in a Picard-Vessiot extension $(\cE,\overline{\partial})$ of $\overline{\cF}$ for $(P-\lambda)(y)=0$, whose field of constants is $\ccC$ and whose derivation $\overline{\partial}$ is defined by $P-\lambda$. The natural extension of $Q-\mu$ to $(\cE,\overline{\partial})$ allows to consider the action of $Q-\mu$ on the $\ccC$-linear space of the solutions of $(P-\lambda)(y)=0$. 
    
    Since $P-\lambda$ and $Q-\mu$ commute, then $(Q-\mu)(\psi_i)$ are solutions of $P-\lambda$. Therefore, there exists an $n\times n$ matrix $M$ with coefficients in $\ccC$ such that
    \[ W((Q-\mu)(\psi_i))=M W(\psi_i).\]
    Taking determinants and using Poisson's formula in Theorem \ref{thm-Poisson}, we obtain
    \[\dres(P-\lambda,Q-\mu)=\frac{\det W((Q-\mu)(\psi_i))}{ \det W(\psi_i)}=\det M\in \ccC.\]
    Thus $\dres(P-\lambda,Q-\mu)\in \Sigma[\lambda,\mu]\cap \ccC=\coC [\lambda,\mu]$.
\end{proof}

Let us consider $(\lambda_0, \mu_0)\in \coC^2$. For a commuting pair $P$, $Q$, 
the next corollary characterizes the existence of common solutions of the eigenvalue problem
\begin{equation}\label{eq-eproblem0}
    Py=\lambda_0 y,\,\,\, Qy=\mu_0 y.
\end{equation}
It follows from Theorem \ref{thm-DR} and the fact that 
$h(\lambda_0,\mu_0)=\dres(P-\lambda_0,Q-\mu_0).$

\begin{cor}\label{cor-Sol_l0m0}
Given commuting differential operators $P$ and $Q$ in {$\Sigma [\partial]\backslash \coC[\partial]$, both of positive order}, the spectral problem \eqref{eq-eproblem0} has a non trivial (common) solution in a Picard-Vessiot extension of $\Sigma$ for $P-\lambda_0$ (or $Q-\mu_0$) if and only if $h(\lambda_0,\mu_0)=0$.
\end{cor}

\subsection{Computing the Burchnall\textendash Chaundy ideal of a pair}\label{sec-generator}

 It is ensured by Theorem \ref{thm-Prev} that the differential resultant $h(\lambda,\mu)=\dres(P-\lambda,Q-\mu)$ is a polynomial in $\coC [\lambda,\mu]$. Resultants are also called eliminants and this is the feature of resultants we will emphasize next.
Let us consider the left ideal generated by $P-\lambda$ and $Q-\mu$ in $\Sigma[\lambda,\mu][\partial]$
\begin{equation}\label{def-idealPQ}
(P-\lambda,Q-\mu)=\{C (P-\lambda)+D(Q-\mu)\mid C,D\in \Sigma[\lambda,\mu][\partial]\}    
\end{equation}
and the elimination ideal 
\begin{equation}\label{eq-elim}
\cE (P-\lambda,Q-\mu):=(P-\lambda,Q-\mu)\cap  \Sigma[\lambda,\mu],  
\end{equation}
which is a                     
                               
                    two sided ideal of $\Sigma[\lambda,\mu]$, 
and
\begin{equation}\label{eq-elimC}
\cE_{\coC}(P-\lambda,Q-\mu):=(P-\lambda,Q-\mu)\cap  \coC [\lambda,\mu],   
\end{equation}
which is a two sided ideal of $\coC [\lambda,\mu]$.
Observe that by Lemma \ref{lem-linComb}, with differential domain of coefficients $\Sigma[\lambda,\mu]$, and by Theorem \ref{thm-Prev}, it follows that
\[h(\lambda,\mu)=\dres(P-\lambda,Q-\mu)\in  \cE_{\coC} (P-\lambda,Q-\mu). \]
Thus both elimination ideals are nonzero.

\medskip

It was proved by Wilson in \cite{W1985}, that $h(P,Q)=0$, that is $h(\lambda,\mu)\in\BC (P,Q)$, in the case of differential operators $P$ and $Q$ whose coefficients are complex valued smooth functions of $x$ defined in some (real or complex) neighbourhood of $x=0$. We use here the argument of \cite{W1985}, Proposition 5.3 in a more general framework, where differential operators have coefficients in an arbitrary differential field $\Sigma$, with field of constants algebraically closed of zero characteristic. For this purpose we develop the following construction.

\medskip

Considering $\Sigma[\lambda,\mu]$ as a $\Sigma$-vector space with basis $\{\lambda^i\mu^j\}$, we can define the $\Sigma$-linear map
\begin{equation}
    \e: \Sigma [\lambda,\mu] \rightarrow \Sigma[\partial], \mbox{ defined by } 
  \e\left(\sum_{i,j}  \sigma_{i,j} \lambda^i \mu^j\right)=\sum_{i,j}  \sigma_{i,j} \ec\left(\lambda^i \mu^j\right).
\end{equation}
In this setting, given $g\in \Sigma[\lambda,\mu]$ we will also denote $g(P,Q):=\e (g)$.
Thus 
\[\Ker(\e)=\{g\in\Sigma[\lambda,\mu]\mid g(P,Q)=0\}.\]
Observe that the restriction of $\e$ to the subring of constants $\coC [\lambda , \mu ]$ of $\Sigma[ \lambda , \mu ]$ is the ring homomorphism $\ec $ defined in \eqref{def-ePQ}, and also that 
\begin{equation}\label{eq-kernels}
 \BC(P,Q)=\Ker(\ec)=\Ker(\e)\cap \coC[\lambda , \mu ].   
\end{equation}

We will proceed next to calculate a  generator for both kernels. We will begin by demonstrating that the  differential resultant $ \dres(P-\lambda,Q-\mu)$ belongs to the ideal $\BC(P,Q)$, that is, $h$ is a Burchnall\textendash Chaundy polynomial. This is Theorem \ref{thm-hPQ}, and to prove it we will need some auxiliary results. The generator of the ideal $\BC(P,Q)$ is provided in Theorem \ref{thm-BCf}, a prime ideal that is therefore a principal ideal of $\coC[\lambda , \mu ]$.

\begin{lem}\label{thm-EinBC}
Given commuting differential operators $P$ and $Q$ in {$\Sigma [\partial]\backslash \coC[\partial]$, both of positive order}, with the previous notation, it holds that 
\[ \cE (P-\lambda,Q-\mu)\subseteq \Ker(\e).\]
\end{lem}

\begin{proof}
Given a polynomial $g$ in $ \cE (P-\lambda,Q-\mu)$ then
\[g(\lambda,\mu)=C(P-\lambda)+D(Q-\mu),\,\,\, C,D\in \Sigma[\lambda,\mu][\partial].\]
Observe that  the differential operator $g(P,Q)=\e (g)$ has finite order, $q=\ord ( g(P,Q) ) $.

Let  $\lambda_0 $ be a constant in $ \coC$. Since $\coC$ is algebraically closed, 
there exists $\mu_0$ in the nonempty set $\{ \mu\in \coC \mid h(\lambda_0,\mu)= 0 \}$.
By Corollary \ref{cor-Sol_l0m0}, there exists a common eigenfunction $ \psi_{\lambda_0}$ for the coupled spectral problem \eqref{eq-eproblem0}. Consequently, the following is an infinite set of linearly independent eigenfunctions
\[
\Psi=\{\psi_{\lambda_0}\mid \lambda_0\in \coC\},
\]
since eigenfunctions associated to different eigenvalues  are linearly independent. Now, for every $\psi_{\lambda_0}\in\Psi$, it holds that
\[g(P,Q)(\psi_{\lambda_0})=g(\lambda_0,\mu_0)\cdot \psi_{\lambda_0}=C^0(P-\lambda_0)(\psi_{\lambda_0})+D^0(Q-\mu_0)(\psi_{\lambda_0})=0,\]
where $C^0$ and $D^0$ are the result of evaluating the coefficients of $C$ and $D$ in $(\lambda_0,\mu_0)$. Moreover, $\Psi$ is included in the $\coC$-linear space of solutions of the equation $g(P,Q)(y)=0$, whose dimension is $q$.  Then  $g(P,Q)$ is the zero operator.
\end{proof}

\para 

As a consequence of  Lemma  \ref{thm-EinBC}, formula \eqref{eq-kernels} and Theorem \ref{thm-Prev} the following result is proved.

\begin{thm}\label{thm-hPQ} Given commuting differential operators $P$ and $Q$ in {$\Sigma [\partial]\backslash \coC[\partial]$, both of positive order}, then
    \[h(\lambda,\mu)=\dres(P-\lambda,Q-\mu)\in \BC(P,Q).\]
\end{thm}

\medskip

The radical of the ideal $(h)$ generated by $h$ in $\coC [\lambda,\mu]$ equals $(f)$, see for instance \cite{Cox}, pag. 178, Proposition 9, where $f$ is the radical of $h$, that is 
\begin{equation}\label{eq-f-red}
f=\sqrt{h}=h/\gcd(h,h')   
\end{equation}
is the square-free part of $h$.  Finally, the next theorem  provides the generator of the ideal $\BC(P,Q)$.

\begin{thm}\label{thm-BCf}
    Let us consider commuting differential operators $P$ and $Q$ in $\Sigma[\partial]\backslash\coC[\partial]$, both of positive order, and $f=\sqrt{h}$, with $h=\dres(P-\lambda,Q-\mu)$. It holds that $f$ is the irreducible polynomial such that
\[\BC(P,Q)=(f).\]
\end{thm}
\begin{proof}
    Recall that $\coC$ is an algebraically closed field. By Remmark \ref{rem-GammaPQ}, $\Gamma$ is an irreducible algebraic curve in $\coC^2$, meaning that its defining ideal $\BC(P,Q)=(f)$ is the ideal generated by an irreducible polynomial $f$. 
    
    In addition, by Theorem \ref{thm-hPQ} we know that $h\in (f)$ and $(f)$ is a radical ideal. Therefore $f=\sqrt{h}$.
\end{proof}

Given commuting differential operators $P$ and $Q$ in $\Sigma[\partial]\backslash\coC[\partial]$ of positive orders $n=\ord(P)$ and $m=\ord(Q)$, 
by Theorem \ref{thm-BCf}, $h= f^ {\bar{r}}$ for some non zero natural number $\bar{r}$. 
Observe that, by \eqref{eq-drespol}, the degrees in $\mu$ and $\lambda$ of $h$ are respectively
\begin{equation}\label{eq-barn}
  n=  \deg_\mu (h) = \bar{r } \deg_\mu (f) , \quad m=  \deg_\lambda (h) = \bar{r} \deg_\lambda (f) , \quad 
  \end{equation}
then $\bar{r} $ divides $\r:=\gcd(n,m)$, and \ $\bar{n}=\deg_{\mu}(f)\leq n$, $\bar{m}=\deg_{\lambda}(f)\leq m$.

\medskip

Let us consider the pure lexicografic monomial ordering in $\Sigma [\lambda,\mu]$ with $\mu>\lambda$. 

Given $g\in \Sigma [\lambda,\mu]$, 
by the Division Theorem,  see \cite{Cox}, Theorem 3, page 64, we can write $g=qf+g_N$  where $q,g_N\in \Sigma [\lambda,\mu]$, 
\begin{equation}\label{eq-nf}
    g_N(\lambda,\mu)=\alpha_{\bar{n}-1}(\lambda)\mu^{\bar{n}-1}+\cdots +\alpha_1(\lambda)\mu+\alpha_0(\lambda), \mbox{with } \alpha_j\in \Sigma[\lambda],
\end{equation}
since the leading monomial  $LT(f)=\mu^{\bar{n}}$ of $f$ does not divide $g_N$. 
We call $g_N$ the {\sf normal form of $g$ with respect to $f$}, and for short we write {$g_N$ is the normal form of \sf $g$ w.r.t. $f$}.  

\medskip

Let us consider a polynomial $g\in \coC [\lambda,\mu]$, and  the differential ideal $[g]$ generated by $g$ in the differential ring $\Sigma[\lambda,\mu]$, whose field of constants is $\coC [\lambda,\mu]$. Observe that $[g]=\{\ell g\mid \ell\in \Sigma[\lambda,\mu]\}$ and $(\ell g)'=\ell'g$, since $g\in \coC [\lambda,\mu]$. Furthermore $[g]$ is the ideal of $\Sigma [\lambda,\mu]$ generated by $g$, but we use the notation $[g]$  to distinguish it from the ideal $(g)$ generated in $\coC [\lambda,\mu]$ by $g$ .

\begin{lem}\label{lem-Ker}
Given commuting differential operators $P$ and $Q$ in {$\Sigma [\partial]\backslash \coC[\partial]$, both of positive order}, with the previous notation,  it holds that $\Ker (\e)= [f]$.
\end{lem}
\begin{proof}
By Theorem \ref{thm-BCf} we know that $\BC (P,Q)=(f)$.  We observe that $[\BC(P,Q)]=[f]\subseteq \Ker(\e)$.
Let us prove the other inclusion.

Given $g\in \Ker (\e)$, let us consider its normal form $g_N$ w.r.t. $f$ as in \eqref{eq-nf}.  If we denote $d_j=\deg(\alpha_j)$ then $\ord(\alpha_j(P)Q^j)=n d_j +j m$, $j=0,1,\ldots ,\bar{n}-1$. Let us define
\begin{equation*}
    \cO=\{n d_j +j m\mid \alpha_j(P)Q^j\neq 0, j\in \{0,1,\ldots ,\bar{n}-1\}\}.
\end{equation*}
Observe that $\e (g)=\e(qf)+\e(g_N)$ and $\e(qf)=\e(q)\e(f)$, only because $f$ is constant.
Since $g(P,Q)=\e(g)=0$ and $f(P,Q)=\e(f)=\ec(f)=0$ then $H=g_N(P,Q)=0$ is a zero operator. If we assume that at least one of the terms of $H$ is nonzero, that is $\cO\neq \emptyset$ then by \eqref{eq-ord-null}
\[n d_{j_1} +j_1 m=n d_{j_2} +j_2 m\]
for distinct $j_1,j_2\in \{0,1,\ldots ,\bar{n}-1\}$. 
Reorganizing and dividing by $\r=\gcd(n,m)$, it holds
\begin{equation}\label{eq-in1}
    \hat{n}|d_{j_1}-d_{j_2}|=\hat{m}|j_1-j_2|,
\end{equation}
where $n=\hat{n}\r$ and $m=\hat{m}\r$, $\gcd(\hat{n},\hat{m})=1$. By \eqref{eq-barn} we have $\hat{n}\leq \bar{n}\leq n$. 
If $|j_1-j_2|< \hat{n}$ then by \eqref{eq-in1} $|d_{j_1}-d_{j_2}|<\hat{m}$ and $\hat{m}\mid\hat{n}$ contradicting that $\hat{n}$ and  $\hat{m}$ are coprime. If $\hat{n}\leq |j_1-j_2|$ then $|j_1-j_2|=\hat{n}t+\hat{j}$, with $0\leq \hat{j}<\hat{n}$. Hence 
$\hat{n}(|d_{j_1}-d_{j_2}|-\hat{m}t)=\hat{m}\hat{j}$,
implying  $\hat{n}\mid\hat{m}$, which is a contradiction. Therefore we conclude that $\cO=\emptyset$, in other words that $g_N$ is the zero polynomial. This proves that $g=qf$, that is $g\in [f]$.
\end{proof}

\begin{thm}\label{thm-primo}
Let us consider commuting differential operators $P$ and $Q$ in $\Sigma[\partial]\backslash\coC[\partial]$, both of positive order, and $f=\sqrt{h}$, with $h=\dres(P-\lambda,Q-\mu)$. Then
    $[f]$ is a prime differential ideal in $ \Sigma [\lambda,\mu]$.
\end{thm}
\begin{proof}
As defined in \eqref{eq-GammaPQ}, let $\Gamma_{P,Q}$ be the spectral curve of the pair $P,Q$. For each point $(\lambda_0,\mu_0)\in \Gamma_{P,Q}$, by Corollary \ref{cor-Sol_l0m0} there exists a common solution $\psi_0$ in a Picard-Vessiot extension of $\Sigma$ for $P-\lambda_0$.

Let us consider $g_1,g_2 \in  \Sigma [\lambda,\mu]$ and assume that $g_1\cdot g_2\in [f]$. 
Observe that 
\begin{equation*}
    0=g_1(\lambda_0,\mu_0)\cdot g_2(\lambda_0,\mu_0)\in   \Sigma .
\end{equation*}
Thus, we can write the following partition of $\Gamma_{P,Q}$
\begin{equation*}
   \Gamma_{P,Q}=\Gamma_1\cup \Gamma_2 \mbox{ where }\Gamma_i=\{(\lambda_0,\mu_0)\in \Gamma_{P,Q}\mid g_i(\lambda_0,\mu_0)=0\}.
\end{equation*}
Having an algebraically closed field of constants $\coC$ implies that at least one of the two components, say $\Gamma_1$, has an infinite number of points. 

We  then have an infinite set of linearly independent eigenfunctions $\Psi=\{\psi_0\mid (\lambda_0,\mu_0)\in \Gamma_1\}$.
For every $\psi_0\in \Psi$ in we have
\[0=g_1(\lambda_0,\mu_0)\cdot \psi_0=g_1(P,Q)(\psi_0)\]
This contradicts the fact that $g_1(P,Q)$ is a finite order differential operator. This proves that $g_1\in \Ker(\e)$. By Lemma \ref{lem-Ker}, $g_1\in [f]$.
\end{proof}

We proceed next to emphasize the relationship between the elimination ideals defined in \eqref{eq-elim} and \eqref{eq-elimC} and the ideal of Burchnall\textendash Chaundy polynomials of a commuting pair of ODOs.

\begin{thm}\label{thm-BCPQ}
    Let us consider commuting differential operators $P$ and $Q$ in $\Sigma[\partial]\backslash\coC[\partial]$, both of positive order, and $f=\sqrt{h}$, with $h=\dres(P-\lambda,Q-\mu)$.  It holds that 
    \begin{enumerate}
        \item  The radical of the elimination ideal $ \cE_{\coC} (P-\lambda,Q-\mu)$   equals  $\BC(P,Q)=(f)$.
        \item The radical of the elimination ideal $ \cE (P-\lambda,Q-\mu)$  equals $[f]$. 
    \end{enumerate}
\end{thm}

\begin{proof}
By Lemma \ref{thm-EinBC} and Lemma \ref{lem-Ker}  we obtain
\begin{equation*}
    \begin{array}{rl}
         \cE (P-\lambda,Q-\mu)\subseteq & \Ker(\e)=[f], \\
         \cE_{\coC} (P-\lambda,Q-\mu)=\cE (P-\lambda,Q-\mu)\cap \coC [\lambda, \mu]\subseteq & \Ker(\e) \cap \coC [\lambda, \mu]=\BC(P,Q)=(f),
    \end{array}
\end{equation*}
using also equality \eqref{eq-kernels} and Theorem \ref{thm-BCf}. Now, taking radicals, and applying Theorem \ref{thm-primo}, the following inclusions hold
\begin{equation*}
    \begin{array}{rl}
         \sqrt{\cE (P-\lambda,Q-\mu)}
         \subseteq  & \Ker(\e)
         =[f] \\
         \sqrt{\cE_{\coC} (P-\lambda,Q-\mu)}= &\BC(P,Q)=(f),
    \end{array}
\end{equation*}
where the last equality is satisfied since $\coC$ is an algebraically closed field, see  \cite{fulton2008}.

Moreover, the differential resultant $h =f^{\bar{r}} $ is a polynomial in $\cE (P-\lambda,Q-\mu)$, by Theorem 
\ref{thm-hPQ}. Consequently, $f$ belongs to the radical $ \sqrt{\cE (P-\lambda,Q-\mu)}$. Thus, \(  \sqrt{\cE (P-\lambda,Q-\mu )} = [f]\).
\end{proof}

Summarizing, the spectral curve $\Gamma_{P,Q}$ of a commuting pair $P,Q\in\Sigma[\partial]\backslash\coC[\partial]$ is defined by the irreducible polynomial $f=\sqrt{h}$, with $h=\dres(P-\lambda,Q-\mu)$. The commutative $\coC$-algebra $\coC[P,Q]$ is isomorphic to the coordinate ring 
$\coC [\Gamma_{P,Q}]$ of this curve 
\begin{equation*}
    \coC[P,Q]\simeq \coC [\Gamma_{P,Q}]:=
    \frac{\coC[\lambda,\mu]}{\BC(P,Q)}
    =\frac{\coC[\lambda,\mu]}{(f)}.
\end{equation*}
In addition, the prime differential ideal $[f]$ determines a differential domain 
\begin{equation*}    
\Sigma[\Gamma_{P,Q}]:=\frac{\Sigma[\lambda,\mu]}{[f]},
\end{equation*}
whose fraction field, denoted by $\Sigma(\Gamma_{P,Q})$, will play a key role in the study of the coupled spectral problem  
\[
Py=\lambda y,\quad Qy=\mu y. 
\]

\begin{rem}
    In some cases, the centralizer $\cC (L)$ of a differential operator $L$ is the affine ring of a planar curve, as in the case of algebro-geometric Schrödinger operators where $\cC(L)=\coC[L,A]$. 
The spectral Picard-Vessiot extension of $\Sigma(\Gamma_{L,A} )$ for $(L-\lambda)(y)=0$ introduced in \cite{MRZ2}, is a Liouvillian extension of  $\Sigma ( \Gamma_{L,A} )$ determined by the solution of the  first order greatest common right divisor of $L-\lambda$ and $A-\mu$, as differential operators with coefficients in $\Sigma(\Gamma_{L,A})$.
\end{rem}

\section{Centralizers as coordinate rings}
\label{sec-cenCoord}

Let $L\in \Sigma[\partial]\backslash \coC [\partial]$ be a third order  operator, whose centralizer is non trivial. We proved in Lemma \ref{lem-prime3} that $\BC(L)$ is a prime ideal, and in Theorem \ref{thm-CL} that 
\[
\cC(L)\simeq \frac{\coC[\lambda,\mu_1,\mu_2]}{\BC(L)}.
\]
Our next goal is to obtain a computational description of $\BC(L)$. We will use differential resultants for this purpose.

\subsection{Normalized basis}\label{sec-normalizedbasis}

By Theorem \ref{thm-hPQ}, each operator $A$ in the centralizer $\cC(L)$, together with $L$ satisfies the algebraic equation defined by an irreducible polynomial $f_A (\lambda, \mu)\in\coC[\lambda ,\mu ]$ defined by the differential resultant $\dres(L-\lambda,A-\mu)$, that is
\begin{equation*}
    f_A (L,A)=0, \mbox{ with } f_A (\lambda, \mu)=\mu^3 + a_{2}(\lambda) \mu^{2}+  a_{1}(\lambda) \mu+ a_0(\lambda).
\end{equation*}
The Tschirnhaus transformation of $f_A (\lambda, \mu)$ gives a new polynomial 
\begin{equation}\label{eq-Tch}
    f_{ \tilde{A} }(\lambda, \mu)= f_A (\lambda, \mu-\frac{1}{3} a_{2}(\lambda))=\mu^3 + \tilde{a}_{1}(\lambda)\mu+ \tilde{a}_0 (\lambda).
\end{equation}
 satisfied by $L$ and 
the operator 
\begin{equation}\label{eq-Tsch}
    \tilde{A}:= A-\frac{1}{3} a_{2}(L).
\end{equation}
Observe that, by theorems \ref{thm-CPQ} and \ref{thm-BCPQ}
\begin{equation*}
    \frac{\coC [\lambda,\mu]}{(f_A)}\simeq \coC[L,A]=\coC[L,\tilde{A}]\simeq \frac{\coC [\lambda,\mu]}{(f_{\tilde{A}})}.
\end{equation*}
In addition, $f_{ \tilde{A} }\in \BC (L,\tilde{A})$ and by \eqref{eq-Tch} $f_{ \tilde{A} }$ is an irreducible polynomial because $f_A$ is irreducible over $\coC[\lambda,\mu]$. Thus Theorem \ref{thm-BCPQ} implies
\begin{equation*}
    f_{ \tilde{A} }(\lambda, \mu)=\dres(L-\lambda,\tilde{A}-\mu).
\end{equation*}

Recall that any basis of $\cC(L)$ is a minimal order basis, as in Definition \ref{defi-minBasis}.

\begin{defi}
    Let $L\in \Sigma[\partial]\backslash \coC [\partial]$ be a third order operator, whose centralizer is non trivial. We will call $A\in \cC(L)\neq \coC[L]$ a {\sf normalized operator of $\cC(L)$} if $\dres(L-\lambda,A-\mu)$ has no term of degree $2$ in $\mu$. A basis $\cB(L)=\{1,A_1,A_2\}$ of $\cC(L)$ as a $\coC [L]$-module is called a {\sf normalized basis of $\cC(L)$} if $A_1$ and $A_2$ are normalized operators of $\cC(L)$.
\end{defi}

The next result shows that any other basis 
of $\cC(L)$ as a $\coC [L]$-module is determined by a normalized basis.

\begin{thm}\label{thm-NormalB}
    Let $L\in \Sigma[\partial]\backslash \coC [\partial]$ be a third order operator, whose centralizer is non trivial. Let us consider  a normalized basis $\cB(L)=\{1,A_1,A_2\}$ of $\cC(L)$. Then for any other basis $\{1,B_1,B_2\}$ of $\cC(L)$ it holds that  
\begin{equation}\label{eq-basisN}
 B_i=\alpha_i A_i+q_i(L),\,\,\, i=1,2,   
\end{equation}
with  $\alpha_i\in \coC$, $q_i(\lambda)\in \coC[\lambda]$, $  \deg_\lambda (q_i (\lambda ))\leq \frac{\ord A_i -i}{3}$. Moreover, if $\{1,B_1,B_2\}$ is also a normalized basis then 
\[B_i=\alpha_i A_i,\,\,\, i=1,2.\]
\end{thm}
\begin{proof}
We know that $\cC(L)=\coC[L,A_1,A_2]=\coC[L,B_1,B_2]$.
W.l.o.g., let us assume that $o_1=\ord(A_1)<o_2=\ord(A_2)$. 
Since $\{1,B_1,B_2\}$ is a minimal order basis, let us assume that $B_i$ has minimal order congruent with $i$ ($\mod \ 3$). Then $\ord(B_i)=o_i$ and, $o_1<o_2$ together with \eqref{eq-cen_mod},
implies 
\[B_1=\alpha_1 A_1+q_1(L)\mbox{ and }B_2=\alpha_2 A_2+p(L)A_1+q_2(L),\]
with $\alpha_i\in \coC$, $p(\lambda), q_i(\lambda)\in \coC[\lambda]$ and $  \deg_\lambda (q_i (\lambda ))\leq \frac{\ord A_i -i}{3}$.

Let us prove that $p(\lambda)$ is identically zero. We know that $\BC(L,B_2)=(f)$ where $f=\dres(L-\lambda,B_2-\mu)\in \coC[\lambda,\mu]$ is a polynomial of degree $3$ in $\mu$. Thus
\begin{equation}\label{eq-B2}
B_2^3=\beta_2(L)B_2^2+\beta_1(L)B_2+\beta_0(L), \,\,\, \beta_j\in\coC [\lambda].
\end{equation}
Since $\cB(L)$ is a normalized basis it holds that
\begin{equation}\label{eq-A_i3}
    A_i^3=\gamma_{i,1}(L) A_i+\gamma_{i,0}(L),\,\,\, \gamma_{i,k}\in \coC[\lambda].
\end{equation}
Computing $B_2^3=(\alpha_2 A_2+p(L)A_1+q_2(L))^3$, for instance with Maple and using \eqref{eq-A_i3} to replace $A_i^3$, we obtain only the monomials $A_1^2 A_2$ and $A_1 A_2^2$ of degree three in $A_1$ and $A_2$ in \eqref{eq-B23}. Comparing with the r.h.s. of \eqref{eq-B2}, we obtain that the coefficients of the monomials of degree three  must vanish
\begin{equation}\label{eq-B23}
    (\alpha_2 A_2+p(L)A_1+q_2(L))^3=3 \alpha_2 p(L)^2 A_1^2A_2+ 3 \alpha_2^2 p(L) A_1 A_2^2+ \cdots .
\end{equation}
Thus $p(\lambda)$ must be identically zero and we obtain
\[B_2^3=(\alpha_2 A_2+q_2(L))^3=
3 \alpha_2^2 q_2(L) A_2^2+(\alpha_2^3 \gamma_{2,1}(L) +3\alpha_2 q_2(L)^2)A_2+ \alpha_2^3 \gamma_{2,0}(L) +q_2(L)^3.\]
If $\{1,B_1,B_2\}$ is also a normalized basis, the r.h.s. of \eqref{eq-B2} cannot have a term $A_2^2$, then $q_2(\lambda)$ must be the zero polynomial and $B_2=\alpha_2 A_2$. We  analogously prove that $B_1=\alpha_1 A_1$.
\end{proof}

The previous result explains that a normalized basis is unique up to multiplication by constants. We include this result for completion but we do not need to choose the normalized basis in the remaining parts of this work.

\subsection{Generators of the Burchnall\textendash Chaundy ideal of a third order ODO}\label{sec-genBC}

Let us consider  a basis $\cB(L)=\{1,A_1,A_2\}$ of $\cC(L)$. Let us denote
\begin{equation*}
    f_i(\lambda,\mu_i):=\dres(L-\lambda,A_i-\mu_i)=\mu_i^3-\lambda^{o_i}+\cdots \in\coC [\lambda ,\mu_i],
\end{equation*}
if $\cB(L)$ is the normalized basis then
\begin{equation*}
    f_i(\lambda,\mu_i):=\mu_i^3-\gamma_{i,1}(\lambda) \mu_i-\gamma_{i,0}(\lambda).
\end{equation*}
Recall that $f_i$ are irreducible and $\BC(L,A_i)=(f_i)$, $i=1,2$, see Section \ref{sec-generator}.
In addition let as denote by $f_3$ the irreducible polynomial in $\coC[\mu_1,\mu_2]$ obtained as the radical of 
\begin{equation*}
    \dres(A_1-\mu_1,A_2-\mu_2)=\mu_2^{o_1}-\mu_1^{o_2}+\cdots.
\end{equation*}
We have $\BC(A_1,A_2)=(f_3)$.

\medskip

The situation so far is described by the following chain of ideals in $\coC[\lambda,\mu_1,\mu_2]$ and we will determine first if the last two inclusions are identities
\begin{equation}\label{eq-inclusions}
    (0)\subset (f_i)\subset (f_1,f_2)\subseteq (f_1,f_2,f_3) \subseteq \BC(L),\,\,\, i=1,2.
\end{equation}
Let us consider the following algebraic varieties defined by the previous ideals
\begin{equation*}
 \Gamma:=V(\BC(L)) \quad , \quad    \gamma: =V(f_1,f_2,f_3) \quad , \quad 
 \beta :=V(f_1,f_2).
\end{equation*}
By the inclusions in \eqref{eq-inclusions}, we obtain $ \Gamma \subseteq \gamma\subseteq\beta$. Observe that $\beta = V(f_1 ) \cap V (f_2 )$ is the intersection of the irreducible surfaces  
defined by $f_1(\lambda,\mu_1)=0$ and $f_2(\lambda,\mu_2)=0$, therefore $\beta$ is a space algebraic curve. 



The Zariski closure of the projection $\overline{\Pi_{\lambda}(\beta)}$ of $\beta$ onto the plane $\lambda=0$ is a plane algebraic curve defined by the square free part $r(\mu_1,\mu_2)$ of the algebraic resultant $\res_{\lambda} (f_1,f_2)$ w.r.t. $\lambda$, see  \cite{Cox}, Chapter 3, \S 2 and \cite{AB}, Section 2. This projected curve $\overline{\Pi_{\lambda}(\beta)}=V(r)$ could be irreducible or not. Observe that  $r(A_1,A_2)=0$ and therefore $r\in\BC(A_1,A_2)=(f_3)$. Thus there are two possible situations:
\begin{enumerate}
    \item {\bf Irreducible.} If $r(\mu_1,\mu_2)$ is irreducible then $(f_1,f_2)=(f_1,f_2,f_3)$ and $V(r)$ is an irreducible curve.

    \item  {\bf Non irreducible.} If $r(\mu_1,\mu_2)$ is not irreducible then $f_3$ properly divides $r$ and $(f_1,f_2)\neq (f_1,f_2,f_3)$. Including $f_3$ in the ideal allows to select one irreducible component of $V(r)$. The example in Section \ref{sec-Examples} illustrates this situation.
\end{enumerate}

We will prove next that $(f_1,f_2,f_3)$ is the defining ideal of an irreducible space curve, in both situations. By Remark \ref{rem-GammaPQ}, the irreducible component of $\Pi_{\lambda}(\beta)$ determined by $f_3$ is a proper curve, it cannot be a point.

\begin{rem}\label{rem-Grobner}
Let us assume that the projection over the plane $\lambda=0$ of the algebraic curve $\beta$, and therefore $\gamma$,  is birational. Note that this holds for almost all projections (see \cite{fulton2008}[Fulton], p155) and hence we can make this assumption w.l.o.g. More precisely, a valid projection direction can be achieved by an affine change of coordinates, see \cite{AB}, Section 2, which establishes an isomorphism between  coordinate rings, see for instance \cite{SWP}, Theorem 2.24.

A Gr\" obner basis 
$\cF=\{F_0, F_1,\ldots ,F_s\}\subset\coC[\lambda,\mu_1,\mu_2]$ of $(f_1,f_2,f_3)$ w.r.t. the pure lexicographic monomial ordering with $\lambda>\mu_1>\mu_2$ verifies:
\begin{enumerate}
    \item The polynomial $F_0\in \coC[\mu_1,\mu_2]$ is an implicit representation of the Zariski closure of  the projection $\overline{\Pi_{\lambda}(\gamma)}$ of $\gamma$ onto the plane $\lambda=0$. Furthermore, $F_0=f_3$ because $f_3$ is irreducible over $\coC$. Thus this projection is an irreducible algebraic curve $\overline{\Pi_{\lambda}(\gamma)}=V(f_3)$.

    \item On the other hand, since this projection is assumed birational on $\gamma$ then $\cF$ contains a linear polynomial in $\lambda$, say $F_1=g_2(\mu_1,\mu_2)\lambda-g_1(\mu_1,\mu_2)$.
\end{enumerate} 
\end{rem}

\para

Let us consider the pure lexicographic monomial ordering with $\lambda>\mu_1>\mu_2$ in $\Sigma [\lambda,\mu_1,\mu_2]$.
Given $g\in \Sigma[\lambda, \mu_1,\mu_2]$, by the Division Theorem \cite{Cox}, Theorem 3, page 64, we can write
\begin{equation}\label{eq-normalform_Grobner}
    g=\sum_j q_j F_j+g^{\cF}
\end{equation}
where $q_j, g^{\cF}\in\Sigma [\lambda, \mu_1,\mu_2]$. We call $g^{\cF}$  {\sf the normal form of $g$ w.r.t $\cF$}.

\begin{lem}\label{lem-Grobner}
Let us consider a Gr\" obner basis 
$\cF$ of $(f_1,f_2,f_3)$ w.r.t. the pure lexicographic order with $\lambda>\mu_1>\mu_2$.
    The normal form $g^{\cF}$ of $g\in \Sigma[\lambda, \mu_1,\mu_2]$  w.r.t $\cF$ is a polynomial in $\Sigma[\mu_1,\mu_2]$ whose degree in $\mu_1$ is less than $\deg_{\mu_1}(f_3)$.
\end{lem}
\begin{proof}
By the Division Theorem, no monomial of $g^{\cF}$ is divisible by the leading terms  $LT(F_1)=\lambda$ and $LT(F_0)=LT(f_3)$, thus the result follows.
\end{proof}

We are ready to prove that $\BC(L)\subseteq (f_1,f_2,f_3)$.

\begin{thm}\label{thm-BCL} 
Let $L$ be a third order operator  in $ \Sigma[\partial]\backslash \coC [\partial]$, whose centralizer is non trivial.
Given a basis $\{1,A_1,A_2\}$ of the centralizer, let $f_i$, $i=1,2,3$ be the irreducible polynomials over $\coC$ such that $\BC(L,A_i )=(f_i)$ and $\BC(A_1,A_2)=(f_3)$.
Then
    $\BC(L)=(f_1,f_2,f_3)$
\end{thm}
\begin{proof}
Given $g\in \BC(L)$, let $g^{\cF}$ be its normal form w.r.t $\cF$. Then $g^{\cF}\in \BC(L)$. Furthermore, by Lemma \ref{lem-Grobner} then $g^{\cF}\in \BC(A_1,A_2)\subset \coC[\mu_1,\mu_2]$ and $\deg_{\mu_1}(g^{\cF})<\deg_{\mu_1}(f_3)$. We conclude that $g^{\cF}$ is identically zero, that is $g\in (\cF)=(f_1,f_2,f_3)$.
\end{proof}

\begin{cor}\label{cor-Gamma}
    Let $L\in \Sigma[\partial]\backslash \coC [\partial]$ be a third order operator, whose centralizer is non trivial. Then 
\begin{equation*}
    \Gamma=V(\BC(L))
\end{equation*}
is an irreducible algebraic curve whose defining ideal is $\BC(L)=(f_1,f_2,f_3)$, with $f_i$ defined above. Furthermore
\begin{equation*}
    \cC(L)\simeq \frac{\coC[\lambda,\mu_1,\mu_2]}{(f_1,f_2,f_3)}.
\end{equation*}
\end{cor}

\section{Parametric factorization of algebro-geometric ODOs}\label{sec-factoring-Bsq}

Let $L$ be a third order operator  in $\Sigma[\partial]\backslash \coC [\partial]$, whose centralizer is non trivial. We consider next the factorization of $L-\lambda$, for $\lambda$ an algebraic parameter over $\Sigma$. We know that $\lambda$ is not a free parameter, it is governed by the spectral curve $\Gamma$ of $L$. Thus
\[L-\lambda \in \Sigma[\lambda][\partial]\subset \Sigma[\lambda,\mu_1,\mu_2][\partial].\]
{Recall that $\Sigma[\lambda,\mu_1,\mu_2]$ is a differential ring with the extended derivation $\partial$, see the notation in Section \ref{sec-intro}.}

Let us consider a basis $\{1,A_1,A_2\}$ of $\cC(L)$ and the irreducible polynomials $f_i$, $i=1,2,3$ in $\coC [\lambda,\mu_1,\mu_2]$ such that $\BC(L,A_1)=(f_i)$ and $\BC(A_1,A_2)=(f_3)$. By Theorem \ref{thm-BCL}, 
\begin{equation}\label{eq-diffBCL}
    [\BC(L)]=[f_1,f_2,f_3].
\end{equation}
As differential operators in $\Sigma(\lambda,\mu_1,\mu_2)[\partial]$, the pairs $L-\lambda$ and $A_i-\mu_i$ are right coprime, since their differential resultants are nonzero, by Theorem \ref{lem-elimIdeal}. To consider the factorization of $L-\lambda$ we need an appropriate differential field of coefficients.


\medskip

\subsection{Coefficient field for factorization}\label{sec-DiffField}

Associated to $L$ we have defined in \eqref{def-BCL} the prime ideal $\BC(L)$, see Lemma \ref{lem-prime3}. We will prove in this section  that the differential ideal $[\BC(L)]$ is a prime ideal in $\Sigma [\lambda,\mu_1,\mu_2]$. Thus we can define the differential domain
\begin{equation}\label{eq-anillocurva}
    \Sigma[\Gamma]=\frac{\Sigma [\lambda,\mu_1,\mu_2]}{[\BC(L)]}
\end{equation}
and its fraction field $\Sigma(\Gamma)$, which therefore is a differential field.

\medskip

Considering $\Sigma[\lambda,\mu_1,\mu_2]$ as a $\Sigma$-vector space with basis 
\[\{M_{\alpha}=\lambda^{\alpha_0}\mu_1^{\alpha_1}\mu_2^{\alpha_2}\mid \alpha=(\alpha_0,\alpha_1,\alpha_2)\in\bbN^3\},\]
we can define the $\Sigma$-linear map
\begin{equation}
    \veL: \Sigma [\lambda,\mu_1,\mu_2] \rightarrow \Sigma[\partial], \mbox{ defined by } 
  \veL\left(\sum_{\alpha}  \sigma_{\alpha} M_{\alpha}\right)=\sum_{\alpha}  \sigma_{\alpha} \eL\left(M_{\alpha}\right),
\end{equation}
where $\eL$ is the ring homomorphism defined in \eqref{eq-eL}. Observe that
given $g=\sum_{\alpha} \sigma_{\alpha} M_{\alpha}\in \Sigma [\lambda,\mu_1,\mu_2]$  and $F\in \coC [\lambda,\mu_1,\mu_2]$ then
\begin{equation}\label{eq-gF}
    \veL (g F)=\sum_{\alpha} \sigma_{\alpha}\eL ( M_{\alpha} F)=\veL(g)\eL(F).
\end{equation}

\begin{lem}\label{lem-veLgF}
Let us consider a Gr\" obner basis 
$\cF$ of $(f_1,f_2,f_3)$ w.r.t. the pure lexicographic order with $\lambda>\mu_1>\mu_2$.
Given $g\in \Ker(\veL)$, the normal form $g^{\cF}$ of $g$ w.r.t $\cF$ verifies
\[g^{\cF}\in [\BC(A_1,A_2)]=[f_3].\]
\end{lem}
\begin{proof}
    With notation as in \eqref{eq-normalform_Grobner}, by \eqref{eq-gF}  we have,
    \[0=\veL(g)=\sum_j \veL(q_j) \eL(F_j)+\veL(g^{\cF})=\veL(g^{\cF}).\]
    Thus $g^{\cF}\in \Ker(\veL)$ and by Lemma \ref{lem-Grobner}, we know that $g^{\cF}\in \Sigma [\mu_1,\mu_2]$. Therefore, by Lemma \ref{lem-Ker}, 
    \[g^{\cF}\in \Ker (\varepsilon_{A_1,A_2})=[\BC(A_1,A_2)]=[f_3].\]
\end{proof}

\begin{lem}
With the previous notation, it holds that $\Ker(\veL)=[\BC(L)]$.
\end{lem}
\begin{proof}
    We will show next that the inclusion of $[\BC(L)]=[f_1,f_2,f_3]$ in $\Ker (\veL)$ is natural. Since $f_i$ has constant coefficients, the next differential  ideal  coincides with the ideal generated by $f_1, f_2, f_3$ in $\Sigma[\lambda,\mu_1,\mu_2]$
    \[[f_1,f_2,f_3]=\{g_1f_1+g_2f_2+g_3f_3\mid g_i\in \Sigma [\lambda,\mu_1,\mu_2]\}.\]
    In addition by \eqref{eq-gF}, $\veL (g_i f_i)=\veL(g_i)\eL(f_i)=0$,
    Which proves that $[\BC(L)\subseteq \Ker(\veL)]$.

    Conversely, given $g\in \Ker(\veL)$, by Lemma \ref{lem-veLgF}, $g^{\cF}\in [f_3]$.
    In addition  by Lemma \ref{lem-Grobner},\break  $\deg_{\mu_1}(g^{\cF})<\deg_{\mu_1}(f_3)$.
    Therefore $g^{\cF}$ is identically zero, proving that $g\in [\cF]=[\BC(L)]$.  
\end{proof}

\begin{thm}\label{thm-BCL-prime}
Let $L$ be a third order operator  in $ \Sigma[\partial]\backslash \coC [\partial]$, whose centralizer is non trivial. Then the differential ideal 
    $[\BC(L)]$ is a prime ideal in $\Sigma [\lambda,\mu_1,\mu_2]$.
\end{thm}
\begin{proof}
Given $g\in [\BC(L)]$, let us assume that $g=g_1\cdot g_2$, with $g_1,g_2\in \Sigma [\lambda,\mu_1,\mu_2]$. 
By Lemma \ref{lem-veLgF}, $g^{\cF}\in Ker(\veL)$. By Lemma \ref{lem-Grobner}, $g_i^{\cF}\in \Sigma [\mu_1,\mu_2]$ and using \eqref{eq-normalform_Grobner} we can write
\[g_i=\Delta_i+g_i^{\cF},\,\,\, \mbox{ with }\Delta_i\in [\BC(L)].\]
Thus
\[g_1\cdot g_2=\Delta_1 \Delta_2+\Delta_1 g_2^{\cF}+\Delta_2 g_1^{\cF}+g_1^{\cF}g_2^{\cF}.\]
Therefore,   $0=\veL(g)=\veL(g_1\cdot g_2)=\veL (g_1^{\cF}g_2^{\cF})$.
This implies that $g_1^{\cF}g_2^{\cF}\in [f_3]$ which is a prime ideal in $\Sigma [\mu_1,\mu_2]$, by Theorem \ref{thm-primo}. We can conclude that  $g_i^{\cF}\in [f_3]$, which shows that $g_i\in [\BC(L)]$, for $i=1$ or $i=2$, proving that $[\BC(L)]$ is a prime ideal.
\end{proof}

We are now ready to work over the field $\Sigma (\Gamma)$, the fraction field of the domain $\Sigma [\Gamma]$ defined in \eqref{eq-anillocurva}. Regarding the differential structure of $\Sigma (\Gamma)$,
we consider the standard 
differential structure of the quotient
ring $\Sigma [\Gamma]$ given by a derivation $\tilde{\partial}$ defined by:
\begin{equation}\label{eq-tildeparial}
    \tilde{\partial}(q + [\BC(L)] := \partial(q) + [\BC(L)]),\,\,\, q\in \Sigma[\lambda,\mu_1,\mu_2].
\end{equation}
Observe that $\tilde{\partial}$ is a derivation in $\Sigma [\Gamma]$ because $[\BC(L)]$ is a differential ideal. By
abuse of notation we will denote by $\partial$ the derivation $\tilde{\partial}$ and its extension to the fraction field  $\Sigma (\Gamma)$.

\subsection{The intrinsic right factor}\label{sec-intrinsicRFactor}

We will study next the factorization of $L-\lambda$ as a differential operator with coefficients in the differential field $(\Sigma(\Gamma),\partial)$ as defined in Section \ref{sec-DiffField}.

\medskip

As differential operators in $\Sigma(\Gamma)[\partial]$, $L-\lambda$ and $A_i-\mu_i$, $i=1,2$ have a non trivial common factor, namely their monic greatest common right divisor $\cL_i:=\gcrd(L-\lambda,A_i-\mu_i)$. For instance, this follows from the Resultant Theorem \ref{thm-DR}, since $f_i=\dres(L-\lambda,A_i-\mu_i)$ is zero in $\Sigma(\Gamma)$. Similarly, for some positive integer $r$, it holds that $f_3^{r}=\dres (A_1-\mu_1,A_2-\mu_2)$, which is zero in $\Sigma(\Gamma)$, implying that $\cL_3:=\gcrd(A_1-\mu_1,A_2-\mu_2)$ is  a differential operator of order greater or equal than one.

We will prove next that $\cL_1$ and $\cL_2$ are indeed differential operators of order $1$ in $\Sigma(\Gamma)[\partial]$. Moreover, the goal of this section is to show that $\cL_1$ and $\cL_2$ are equal as differential operators in $\Sigma(\Gamma)[\partial]$, providing a right factor of $L-\lambda$ that is intrinsic to the nature of $L$. This factor is linked to the hypothesis of having a nontrivial centralizer. 

\medskip

We chose differential subresultants to compute greatest common right divisors, because  they have fairly explicit expressions for their computation, that also allow us to obtain some important theoretical conclusions.  
As defined in Appendix \ref{app-subresultant}, let us consider the first differential subresultants
\begin{align*}
    &\sdres_1 (L-\lambda,A_i-\mu_i)=\phi_{i,0}+\phi_{i,1}\partial ,\,\,\, i=1,2,\\
    &\sdres_1 (A_1-\mu_1,A_2-\mu_2)=\phi_{3,0}+\phi_{3,1}\partial 
\end{align*}
where, for $j=0,1$
\begin{align}\label{eq-phiij}
    &\phi_{i,j}(\lambda,\mu_i ):=\det(S_1^j (L-\lambda ,A_i-\mu_i  )), i=1,2\\
    &\phi_{3,j}(\mu_1 ,\mu_2 ):=\det(S_1^j (A_1-\mu_1  ,A_2-\mu_2 )).
\end{align}
Observe that the coefficients $\phi_{i,j}$ are nonzero polynomials in $\Sigma [\lambda,\mu_1  ,\mu_2 ]$, but we need to check if they are nonzero in $\Sigma(\Gamma)$. Recall that from Section 
\ref{sec-centralizers}, $o_i=\ord(A_i)\equiv i (\mod \ 3 )$.

\begin{lem}\label{lem-phinonzero}
With the notation established above. The class  $\phi_{i,j}+[\BC(L)]$, $i=1,2$, $j=0,1$ in $\Sigma(\Gamma)$ is non zero. In addition, if $\gcd(o_1,o_2)=1$ then the class of $\phi_{3,j}+[\BC(L)]$ is non zero.
\end{lem}
\begin{proof}
    By the construction of the matrices $S_1^j$ as in \eqref{eq-Sij}, the degree in $\mu_i$ of $\phi_{i,j}$ is less than  $3$, for $i=1,2$. 
    Thus $\phi_{i,j}$, $i=1,2$ does not belong to $[\BC(L)]$ because otherwise it is included in
    \[[\BC(L)]\cap \Sigma [\lambda,\mu_i]=[\BC(L,A_i)]=[f_i],\]
    which is not possible.
    Similarly, if $\gcd(o_1,o_2)=1$ then the degree in $\mu_1$ of $\phi_{3,j}$ is less than $o_2$ and $\phi_{3,j}$ does not belong to $[\BC(L)]$ because otherwise, it is included in
    \[[\BC(L)]\cap \Sigma [\mu_1,\mu_2]=[\BC(A_1,A_2)]=[f_3],\]
    observe that in the case $\gcd(o_1,o_2)=1$ then $f_3=\dres (A_1-\mu_1,A_2-\mu_2)$ has degree $o_2$ in $\mu_1$.
\end{proof}

By Lemma \ref{lem-phinonzero}, $\sdres_1 (L-\lambda,A_i-\mu_i)$ are nonzero differential operators in $\Sigma(\Gamma)[\partial]$ of order one. 
The First Subresultant Theorem \ref{thm-FST}  implies that $\cL_i=\gcrd(L-\lambda,A_i-\mu_i)$ equals the first subresultant. 
We can make them monic and write
\begin{equation}\label{eq-phii}
    \cL_i=\gcrd(L-\lambda,A_i-\mu_i)=\partial+{\phi_i+[\BC(L)]},\,\,\,  \mbox{ with }\phi_i:=\frac{\phi_{i,0}}{\phi_{i,1}}.
\end{equation}

\medskip

The next lemma will be important to prove that the right factors $\cL_i$ coincide over the field of the spectral curve $\Sigma(\Gamma)$. For this purpose, let us define the elimination ideal
\begin{equation}\label{eq-EL}
    \cE(L):=(L-\lambda,A_1-\mu_1 , A_2-\mu_2 )\cap \Sigma[\lambda,\mu_1 , \mu_2 ].
\end{equation}

\begin{prop}\label{lem-ELBC}
With the previous notation, it holds  that   $\cE(L)\subset [\BC(L)]$.
\end{prop}
\begin{proof}
Given a polynomial $g$ in $ \cE (L)$ then
\begin{equation}\label{eq-gEL}
g(\lambda,\mu_1,\mu_2)=C(L-\lambda)+D_1(A_1-\mu_1)+D_2(A_2-\mu_2),\,\,\, C,D_1,D_2\in \Sigma[\lambda,\mu_1,\mu_2][\partial].
\end{equation}

With the notation in Remark \ref{rem-Grobner} and by Corollary \ref{cor-Gamma}, every point of $\Gamma$ is  determined by a point $(\eta_1,\eta_2)$ of the projection of the curve $\Gamma$ onto the plane $\lambda=0$, in the following way
\begin{equation}\label{eq-lift}
  \left(\frac{g_1(\eta_1,\eta_2)}{g_2(\eta_1,\eta_2)}, \eta_1,\eta_2\right)\mbox{ with }f_3(\eta_1,\eta_2)=0,  
\end{equation}
except for a finite number of points such that $g_2(\eta_1,\eta_2)=0$.

As in Remark \ref{rem-Grobner}, let us consider a Gr\" obner basis 
$\cF=\{F_0, F_1,\ldots ,F_s\}\subset\coC[\lambda,\mu_1,\mu_2]$ of $(f_1,f_2,f_3)$ w.r.t. the pure lexicographic monomial ordering with $\lambda>\mu_1>\mu_2$. 
Recall that $F_1=g_2(\mu_1,\mu_2)\lambda-g_1(\mu_1,\mu_2)$, with $g_i\in\coC [\mu_1,\mu_2]$. 
By \eqref{eq-lift}, since $\BC(L)=(f_1,f_2,f_3)=(\cF)$ we have
\begin{equation}\label{eq-F1}
    F_1(L,A_1,A_2)=g_2(A_1,A_2)L-g_1(A_1,A_2)=0.
\end{equation}

Given $\eta_1\in \coC$, there exists $\eta=(\eta_1,\eta_2)\in \Pi_{\lambda}(\Gamma)$, the Zariski closure of the projection of $\Gamma$ onto the plane $\lambda=0$. Let $\cE_1$ be the Picard-Vessiot extension of $\Sigma$ for $A_1-\eta_1$. By Corollary \ref{cor-Sol_l0m0}, since $f_3(\eta_1,\eta_2)=0$ there exists $\psi_1\in \cE_1$ a common eigen function of the spectral problem 
\begin{equation}\label{eq-epi}
    A_1 y=\eta_1 y,\,\,\, A_2 y=\eta_2 y.
\end{equation}
Thus $\Psi:=\{\psi_1 \mid \eta_1\in \coC\}$ is an infinite set of linearly independent eigenfunctions for \eqref{eq-epi}.
Observe that, by \eqref{eq-F1}
\begin{equation*}
    0=(L g_2(A_1,A_2)-g_1(A_1,A_2))(\psi_1)
    =L (g_2(\eta)\cdot \psi_1)-g_1(\eta)\cdot \psi_1
\end{equation*}
thus 
\begin{equation}\label{eq-Leta}
    L(\psi_1) = \frac{g_1(\eta)}{g_2(\eta)}\cdot \psi_1.
\end{equation}

To finish the argument, for every $\psi_1\in \Psi$, by \eqref{eq-epi}  and \eqref{eq-Leta}
\begin{align}
    &g(L,A_1,A_2)(\psi_1)=\\
    &g\left(\frac{g_1(\eta)}{g_2(\eta)}, \eta_1,\eta_2\right)\cdot \psi_1=
\left(\Tilde{C}\left(L-\frac{g_1(\eta)}{g_2(\eta)}\right)+\Tilde{D}_1(A_1-\eta_1)+\Tilde{D}_2(A_2-\eta_2) \right)(\psi_1)=0,
\end{align}
\[\]
using \eqref{eq-gEL} and obtaining $\Tilde{C}$, $\Tilde{D}_1$ and $\Tilde{D}_2$ in $\Sigma[\partial]$.
Since $\Psi$ is an infinite set, we can conclude that $g(L,A_1,A_2)=0$, which proves the result.

\end{proof}

Let us denote by $\overline{\phi}_i$ the class  $\phi_i+[\BC(L)]$ in $\Sigma(\Gamma)$. Looking at these functions as representatives of objects in $\Sigma(\Gamma)$, we can establish the following result.

\begin{lem}\label{thm-phi-global} 
With the previous notation. Let us consider $\phi_i$ as in \eqref{eq-phii} then 
\begin{equation}\label{eq-phi-global}
\phi: =\overline{\phi}_1=\overline{\phi}_2\in \Sigma(\Gamma). 
\end{equation}
If $\gcd(o_1,o_2)=1$ then $\phi=\overline{\phi}_3$.
\end{lem}
\begin{proof}
By the definition of subresultants in Appendix \ref{app-subresultant}, we know that 
\[\cL_i=\partial+\phi_i \in (L-\lambda, A_i-\mu_i).\]
The next one is a differential polynomial in $\Sigma[\lambda,\mu_1 , \mu_2 ]$
\[g=\phi_{1,1}\phi_{2,1} (\cL_1-\cL_2)=
\phi_{1,0}\phi_{2,1}-\phi_{2,0}\phi_{1,1}.\]
Observe that it belongs to the elimination ideal $\cE(L)$ as defined in \eqref{eq-EL}. 
In addition, by Lemma \ref{lem-ELBC}, 
\[g\in \cE(L)\subseteq [\BC(L)].\]
Therefore $\overline{g}\equiv 0$ in $\Sigma(\Gamma)$. 
This proves \eqref{eq-phi-global}. 
\end{proof}

We are ready to state an important application of the results of this paper.

\begin{thm}\label{thm-factor}
Let $L$ be a third order operator  in $ \Sigma[\partial]\backslash \coC [\partial]$, whose centralizer is non trivial. Let $\Gamma$ be the spectral curve of $L$. Then $L-\lambda$ has an order one factor $\partial+\phi$ in $\Sigma(\Gamma)[\partial]$.
We call $\partial +\phi $ \  {\sf the intrinsic right factor of $L$}. 
Moreover, given any basis $\{1,A_1,A_2\}$ of $\cC (L)$ as a $\coC[L]$-module, this factor is the greatest common right divisor in $\Sigma(\Gamma)[\partial]$
\begin{equation}\label{eq-globalFac}
  \partial+\phi=\gcrd(L-\lambda,A_1-\mu_1,A_2-\mu_2).  
\end{equation}
\end{thm}
\begin{proof}
    By the First Subresultant Theorem \ref{thm-FST}, 
    \[\partial+\overline{\phi_i}=\gcrd(L-\lambda,A_i-\mu_i), i=1,2\]
    and by Lemma \ref{thm-phi-global} 
    \[\partial+\phi=\partial+\overline{\phi_1}=\partial+\overline{\phi_2}\]
    thus \eqref{eq-globalFac} follows.
\end{proof}

{As a consequence of \eqref{eq-globalFac} and \eqref{eq-phii}
 $\partial+\phi$, the intrinsic right factor of $L-\lambda$ as a differential operator in $\Sigma[\Gamma][\partial]$ can be computed as $\gcrd(L-\lambda,A_i-\mu_i)$, for any $A_i\in \cC(L)$ of minimal order $i\ (\mod \ 3)$ for $i=1$ or $i=2$.}

 \begin{rem}
Let us assume that $L=\partial^3+u_1\partial+u_0$ is in normal form. By direct computation we obtain the factorization 
\begin{equation}\label{eq-intrinsic-rfactor}
    L-\lambda = \left(\partial^2 -\phi\partial+
    u_1 -2\phi' 
    +\phi^2 
    \right)\cdot (\partial +\phi ) ,
\end{equation}
in  $\Sigma(\Gamma)[\partial ]$, under the condition
\begin{equation}\label{eq-relacion}
    \phi^{3}+u_{{1}}
\phi  -3\,\phi  \phi'  -u_{{0}}  + \phi''  +\lambda =0 .
\end{equation}
\end{rem}   

\medskip

\noindent {\bf Factorization at each point $P_0$ of $\Gamma$.}
 The intrinsic right factor is a global factor in the following sense. For every $P_0=(\lambda_0,\eta_1,\eta_2)\in \Gamma\backslash Z$, where $Z$ is a finite number of points in $\coC^3$, we obtain a right factor in $\Sigma[\partial]$ of 
 \[
L-\lambda_0 = N_i \cdot (\partial+\phi_i (P_0 )), \ i=1,2
\]
where $\partial+\phi_i(P_0)$ is the greatest common right divisor of $L-\lambda_0$ and $A_i-\eta_i$ in $\Sigma[\partial]$ and $\phi_i(P_0)$ is the result of replacing $(\lambda, \mu_1,\mu_2)$ in $\phi_i$ by $(\lambda_0,\eta_1,\eta_2)$. {The set of points $Z$ to be removed can be described as }
\[{Z=\{ (\lambda_0,\eta_1,\eta_2)\in\coC^3 \mid (\lambda_0,\eta_i)\in Z_i, i=1,2\}}\]
{where $Z_i=\{(\lambda_0,\eta_i)\in \coC^2\mid\phi_{i,1}(\lambda_0,\eta_i)=0\}$, i=1,2, are finite sets of points, which is a consequence of \eqref{eq-phiij} together with Remark \ref{rem-Grobner}. Furthermore $\partial+\phi_1 (P_0 )=\partial+\phi_2 (P_0 )$ is the greatest common right divisor of $L-\lambda_0$, $A_1-\eta_1$ and $A_2-\eta_2$ for every $P_0=(\lambda_0,\eta_1,\eta_2)\in \Gamma\backslash Z$.}

\begin{rem}\label{rem-eigenring}
    The factorization of an ordinary differential operator with coefficient in a differential field $(\Sigma=\coC(x),\partial=d/dx)$ can be achieved by means of the so called eigenring of $L$, defined by M. Singer in \cite{singer1996testing} as
    \[\cE(L)=\{\overline{R}\in \Sigma[\partial]/\Sigma[\partial]\cdot L\mid LR \mbox{ is divisible by } L \},\]
    where $\overline{R}$ denotes the equivalence class of $R\in \Sigma[\partial]$ in the module $\Sigma[\partial]/\Sigma[\partial]\cdot L$. 
    The centralizer $\cC(L)$ can be seen as a subring of the eigenring, but note that the representative of order smaller than $3$ in the eigenring of an element of the centralizer may not belong itself to the centralizer. {In the example of Section \ref{sec-Examples}, } since $\ord(A_1)=4$ then $A_1$ has a representative in $\cE(L)$  that does not belong to the centralizer.

One could study the factorization of $L-\lambda_0$, for any $\lambda_0\in\coC$ by means of the algorithms based on the eigenring, as in \cite{AMW}.  
What this algorithms would not do is to identify the spectral curve. Furthermore, these algorithms would not allow a formal treatment of $\lambda$ as an algebraic parameter over the coefficient field $\Sigma$. See the example in Section \ref{sec-Examples}.
\end{rem}

\section{A space spectral curve for factorization}\label{sec-Examples}

There are many examples of rank $1$ operators whose centralizers are the ring of a plane algebraic curve, see references in \cite{PRZ2019}. In this section we present the first explicit example of a centralizer isomorphic to the ring of a non-planar spectral curve. We explained the computational method used to obtain this example in \cite{HRZ2023}, based on Wilson's almost commuting bases, \cite{W1985}. We use  this example to illustrate the results of this paper.

\medskip


Let us consider the differential operator in $\bbC(x)[\partial]$, where $\partial=\frac{d}{dx}$
\begin{equation}
     L =\partial^3 -\frac{6}{x^{2}} \partial +\frac{12}{x^{3}}+1 . 
\end{equation}
The centralizer of $L$ equals $\cC(L)=\bbC[L,A_1,A_2]$ with
    \begin{align*}
    A_1=& \partial^4 -\frac{8}{x^2 }\partial^2 +\frac{24}{x^3 }\partial -\frac{24}{x^4},\\
    A_2=& \partial^5 -\frac{10}{x^2 }\partial^3 +\frac{40}{x^3 }\partial^2- \frac{80}{x^4 }\partial+\frac{80}{x^5 }.
    \end{align*}
The Burchall-Chaundy ideal of $L$ equals
\[\BC(L)=(f_1,f_2,f_3),\]
generated by the irreducible polynomials 
\begin{align*}
    &f_1(\lambda,\mu_1)=\dres(L-\lambda,A_1-\mu_1)=-\mu_1^3+(\lambda-1)^4,\\
    &f_2(\lambda,\mu_2)=\dres(L-\lambda,A_2-\mu_2)=-\mu_2^3+(\lambda-1)^5,\\ 
    &f_3(\mu_1,\mu_2)=\dres(A_1-\mu_1, A_2-\mu_2)=\mu_2^4-\mu_1^5.
\end{align*}
Since $\res_{\lambda}(f_1,f_2)=p f_3$ with $p= \mu_1^{10}+\mu_1^5\mu_2^4+\mu_2^8$ then $(f_1,f_2)$ is strictly contained in $(f_1,f_2,f_3)$.  A Gr\" obner basis of $\BC(L)$ in $\bbC[\lambda,\mu_1,\mu_2]$ w.r.t. the pure lexicographic order $\lambda>\mu_1>\mu_2$ is
\[\cF=\{f_3(\mu_1,\mu_2), F_1=g_2(\mu_1,\mu_2 )\lambda +g_1(\mu_1,\mu_2 ), \ldots ,F_5\}\]
with $g_1=- \mu_2^4 - \mu_1^2 \mu_2^3$ and $g_2=\mu_2^4$.

\medskip

The curve defined by $\BC(L)$ is a non-planar curve $\Gamma$ parametrized by 
\begin{equation}\label{def-param}
 \aleph (\tau)=(-\tau^3+1,\tau^4,-\tau^5), \tau\in\bbC.   
\end{equation} 
This is the first explicit example of a non-planar spectral curve. 

The first differential subresultants of $L-\lambda$, $A_1-\mu_1$ and $A_2-\mu_2$ pairwise are equal to 
    \[\phi_{i,0}+\phi_{i,1}\partial, i=1,2,3, j=0,1, \]
    see \eqref{eq-phiij}, with
    \begin{equation*}
    \begin{array}{ll}
         \phi_{1,0}=(1 -\lambda)\mu_1 - \frac{4\mu_1}{x^3} +\frac{8(\lambda-1)}{x^4},&   \phi_{1,1}=(\lambda-1)^2-\frac{2\mu_1}{x^2}+4\frac{(\lambda-1)}{x^3},\\
         \phi_{2,0}=(1 -\lambda)^3 - \frac{4(1-\lambda)^2}{x^2} +\frac{8\mu_2}{x^4}, & 
         \phi_{2,1}=(\lambda-1)^3-\frac{4(1-\lambda)^2}{x^2}+\frac{8\mu_2}{x^3},\\
         \phi_{3,0}=-\mu_2 \left(\mu_1^2  + \frac{4\mu_2}{x^3} -\frac{8\mu_1}{x^4}\right), & \phi_{3,1}=\mu_1^3 - \frac{2\mu_2^2}{ x^2} +\frac{ 4\mu_2\mu_1}{x^3}.
    \end{array}
    \end{equation*}
We have $o_1=4$ and $o_2=5$ thus  $\gcd(o_1, o_2)=1$ and
\[\phi=\overline{\phi}_i=\frac{\phi_{0,i}}{\phi_{1,i}}+[\BC(L)],\,\,\, i=1,2,3.\]
Let us call $\phi(\tau)$ to the result of replacing $(\lambda,\mu_1,\mu_2)$ by the parametrization $\aleph (\tau)$ of $\Gamma$ in $\phi$, which equals

 \begin{equation}
        \phi(\tau):=\phi_i(\aleph(\tau))=\frac{-\tau^3x^3 + 2\tau^2 x^2 - 4\tau x + 4}{(\tau^2 x^2 - 2\tau x + 2)x}.
    \end{equation}
Thus
$L-(\tau^3 +1)= \left( \partial^2+\phi(\tau)\partial+\phi(\tau)^2+2\phi(\tau)'-\frac{6}{x^2 } \right)\cdot (\partial+\phi(\tau))$ as in \eqref{eq-intrinsic-rfactor}.
{At every point $P_0=\aleph (\tau_0)$ of the spectral curve $\Gamma$ of $L$ we recover a right factor $\partial+\phi(\tau_0)$, for $\tau_0\neq 0$.}

    Observe that whenever the spectral curve $\Gamma$ is a rational curve with rational parametrization $ {\bf \aleph }(\tau )= (\aleph_1 (\tau),  \aleph_2 (\tau), \aleph_3 (\tau)  )$, as in the previous example, one can consider the  factorization of $ L-\aleph_1 (\tau)$ as a differential operator in $\Sigma(\tau)[\partial]$,  since the algebraic variable $\tau$ with $\partial(\tau)=0$, would be a free parameter over $\Sigma$. To achieve a full factorization, it is a future project to use the parametrized Picard-Vessiot theory introduced by Cassidy and Singer in \cite{CS}, and studied in \cite{MO2018}. It would be interesting to explore how this theory explains the factorization over the field of the spectral curve, using the results on second order operators in \cite{Arr} combined with the intrinsic factorization of Section \ref{sec-intrinsicRFactor}.

\bigskip



\medskip

\noindent{\Large\bf Acknowledgments}

\para 

A sincere acknowledgement to Rafael Hernandez Heredero  for the enlightening conversations on this topic, that contributed to a better understanding of the mechanisms behind integrability. The authors really appreciate the exchange of ideas regarding the non-planar spectral curves with Emma Previato, we are very grateful for the conversations we had on spectral problems and for her constant encouragement to pursuit this project.

The first author is a member of the Research Group ``Modelos ma\-tem\'aticos no lineales". The authors are partially supported by the grant PID2021-124473NB-I00, ``Algorithmic Differential Algebra and Integrability" (ADAI)  from the Spanish MICINN.

\appendixpage

\appendix

\addappheadtotoc

\section{Centralizers of differential operators}\label{app-Centralizers}

Assume in the most general case that $R$ is a commutative differential ring with no nonzero nilpotent elements whose ring of constants $\coC$ is a field of zero characteristic (not necessarily algebraically closed in this section). We will denote by $R [\partial]$ the ring of differential operators with coefficients in $R$. Any nonzero operator $L\in R [\partial ]$ can be uniquely written in the form
\begin{equation}\label{eq-L-general}
    L= u_0 +u_1 \partial + \cdots +u_n \partial^n , \textrm{with } u_i \in R \ \textrm{and } u_n \not=0 .
\end{equation}
We call $n$ the {\sf order of $L$}, denoted by $\ord (L)$, with the convention $\ord (0)=-\infty$. The element $u_n$ is called the {\sf leading coefficient of $L$}.  The order defines a function $\ord : R[\partial ] \rightarrow \bbN \cup \{ -\infty \}$ that satisfies the following properties
\begin{enumerate}
    \item  For nonzero $Q_1 , \dots , Q_\ell \in R[\partial ]$, the following inequality holds
    \begin{equation}\label{eq-ord-max}
        \ord (Q_1 +\cdots +Q_\ell )\leq \max \{ \ord (Q_i ) \ | \ 1\leq i \leq \ell \}.
    \end{equation}
    \item Consider nonzero $Q_1 , \dots , Q_\ell \in R[\partial ]$ with $\ord (Q_i ) \not= \ord (Q_j )$  whenever $i\not=j$. 
    \begin{equation}\label{eq-ord-null}
        \textrm{If } Q_1 +\cdots +Q_\ell =0 ,\quad  \textrm{then } Q_i =0 \quad  \textrm{for all  } i.
    \end{equation}
\end{enumerate}


We reformulate next results of Goodearl in \cite{Good} adapted to our context. 
The centralizer of $L$ in $R[\partial ]$ is the following (unital) subring of $R[\partial ]$ 
\begin{equation*}
    \cC(L)=\{A\in R [\partial]\mid LA=AL\}.
\end{equation*}

In previous notations and assumptions,  the following result follows (Lemma 1.1 in \cite{Good}).
\begin{lem}\label{lemma-ord}
Let $L$ be as in \eqref{eq-L-general} with $u_n$ a unit of $R$. Let $P$ and $Q$ be nonzero operators in $\cC(L)$ with orders $\ell$ and $m$ and leading coefficients $a_\ell$, $b_m$ respectively. Then
\begin{enumerate}
    \item There exists a nonzero constant $\alpha\in \coC$ such that $a_\ell^n =\alpha u_n^\ell$.
    \item \label{item-domain} We have $P\cdot Q \not= 0$, and also
    \begin{equation}\label{eq-ord}
        \ord (P\cdot Q) = \ord (P) +\ord (Q) .
    \end{equation}
    \item \label{item-cts}If $\ell=m$, there exists a nonzero constant $\alpha\in\coC$ such that $a_\ell =\alpha b_\ell$.
\end{enumerate}

\end{lem}

By \ref{item-domain} in the previous lemma we obtain the following consequence.

\begin{prop}\label{coro-domain}
Let $R$ be a commutative differential ring with no nonzero  nilpotent elements, whose ring of constants $\coC$ is a field of zero characteristic. For any $L$ with invertible leading coefficient in $R$, then $\cC(L)$ does not have zero divisors.
\end{prop}

\medskip

Observe that if $L\in\coC [\partial]$ then $\cC (L)=R[\partial]$. In addition the centralizer $\cC (L)$ always contains the ring
\[\coC[L]=\{p(L)\mid p(\lambda)\in \coC[\lambda]\}.\]

\begin{defi}\label{def-nontrivial}
    We will say that $L\in R[\partial]$ has a {\sf non trivial centralizer} if $\coC(L)$ does not equal $\coC[L]$ or $R[\partial]$, that is
\[\coC[L]\subset \cC(L) \subset R[\partial].\]
\end{defi}

From now on we assume that $L\in R[\partial]\backslash \coC[\partial]$ has a nontrivial centralizer $\cC (L)\neq \coC[L]$. We will also assume that 
 $u_n$ is a unit in $R$. We review next the proof of \cite{Good}, Theorem 1.2 which gives a constructive method to compute a basis of $\cC(L)$ as a free $\coC [L]$-module of finite rank.

\medskip

\noindent\textbf{Construction of a basis.} Due to its importance in this paper, we explain next the procedure given in \cite{Good}, Theorem 1.2. Let $L$ be as in \eqref{eq-L-general} with $u_n$ a unit in $R$. Define 
\begin{equation*}
    X:= \left\{
    i\in \{ 0, \dots , n-1 \} \ | \ \exists P\in \cC (L) , \ \textrm{with } \ord (P) \equiv i\ (\mod \ n )
    \right\}.
\end{equation*}
The set $\cO$ of orders of nonzero differential operators in $\cC(L)$ defines a subgroup $H_\cO$ of $\bbZ/n\bbZ$,
\begin{equation*}
 H_\cO =\left\{
 i+n\bbZ \ | \ i\in \cO
 \right\},
\end{equation*}
whose cardinality is $|H_\cO | =Card (X)$. Therefore, $Card (X)$ is a divisor of $n$.

\begin{lem}\label{lem-minBasis}
    With the previous notation. For each $i\in X$ choose $P_i \in \cC (L)$ of minimal order $n_i \equiv i\ (\mod \ n ) $ and $P_0 = 1$. Then the set
\begin{equation}\label{eq-beta}
    \cB :=
    \left\{
    P_i \ | \ i\in X
    \right\}
\end{equation}
is a  basis of the centralizer $\cC (L)$ as a $ \coC [L]-$module.
\end{lem}
\begin{proof}
    First observe that for a sum $\sum_{i\in X} C_i P_i $. $C_i\in \coC[L]$ to be zero, it is necessary that all summands are zero, because of \eqref{eq-ord-null}  applied to $Q_i = C_i P_i$. But then $C_i =0$, by Corollary \ref{coro-domain}. Consequently $\cB$ is a finite family of $\coC[L]$-linearly independent differential operators, and the centralizer $\cC (L)$ contains the free $\coC [L]$-submodule
\begin{equation*}
    W:= \bigoplus_{i\in X} \coC [L]P_i .
\end{equation*}
Let $Q$ be a differential operator in $ \cC(L)$ of order $m$. We will proceed by induction on $m$ to show that $Q$ is in $W$.

For $Q=u\in R$, the equality $LQ-QL=0$ provides a null differential operator whose  coefficient in $\partial^{n-1}$ is $ nu_n \partial(u)$. But, since $u_n$  is a unit in the zero characteristic ring $R$, we obtain $Q=u\in \coC$. Let us now consider an operator $Q$ of positive order $m>0$ and leading coefficient $c_m$. Let $i_m$ be the remainder of the division of $m$ by $n$, then $ m=qn+i_m$. The differential operator $T_Q :=L^q P_{i_{m}}$ is an operator of order $m$ in the centralizer of $L$. Consequently, by Lemma \ref{lemma-ord}, \ref{item-cts} there exists a nonzero constant $\alpha\in\coC$ such that $u_n^q p_{i_{m}} =\alpha c_m$, where $p_{i_{m}}$ is the leading coefficient of $P_{i_{m}}$. Thus, $Q-\alpha ^{-1} T_Q$ is an operator of order $<m$ in $\cC(L)$. According to the induction hypothesis, it is in $W$ and therefore also is $Q$. Consequently  $\cC (L)=W$.
\end{proof}

\begin{defi}\label{defi-minBasis}
    We call a basis of $\cC(L)$ as a $\coC[L]$-module as defined in Lemma \ref{lem-minBasis} a {\sf minimal order basis}.
\end{defi}

\begin{prop}\label{prop-minBasis}
Let $L\in \Sigma[\partial]\backslash \coC [\partial]$ be an operator  whose centralizer is non trivial. Every basis of $\cC(L)$ as a $\coC[L]$-module  is a  minimial order basis.
\end{prop}
\begin{proof}

Observe that given any basis $\cB$ of  $\cC (L)$ as in \eqref{eq-beta} then $\cC (L)$ is the direct sum of the $\coC [L]$-submodules $M_i =\coC [L]B_i$, for all $i\in X$.  In addition, note that the operators in $\coC [L,P_i]=\{q_0(L)+q_1(L)P_i\mid q_0,q_1\in \coC [\lambda]\}$ of order $o_i=\ord(P_i)$ are of the form $q_0(L)+ c P_i$ for some nonzero $c\in \coC$.

Thus a finite subset $\cS$ of $\cC(L)$ is linearly independent if each element belongs to a different subspace $M_i$. To be a generating set of $\cC(L)$ as a $\coC[L]$-module $\cS$ must contain an element in each $M_i$ and it must be of minimal order, since otherwise it will not generate $P_i$. \end{proof}

It was also proved in Goodearl's work \cite{Good}, Corollary 4.4 that $\cC (L)$ is a  commutative ring, as a consequence of being a subring of the centralizer of $L$ in the ring of pseudo-differential operators with coefficients in $R$, see Section \ref{sec-centralizers} of this work.
The next theorem summarizes the main features of the algebraic structure of $\cC (L)$.

\begin{thm}\label{thm-Appcen}
Let $R$ be a commutative differential ring with no  nonzero nilpotent elements whose ring of constants $\coC$ is a field of zero characteristic.  Given a nonzero differential operator $L$ in $R[\partial]$, whose leading coefficient is a unit in $R$, then
\begin{enumerate}
    \item the rank of $\cC(L)$ as a $\coC[L]$-module divides $\ord(L)$,

    \item $\cC (L)$ is a commutative integral domain
\end{enumerate}
\end{thm}

It is an important fact that each operator $A$ in the centralizer $\cC(L)$, together with $L$ satisfy the algebraic equation defined by a polynomial $h_A (\lambda, \mu)\in\coC[\lambda ,\mu ]$. Using the argument given in \cite{Good}, Theorem 1.13, since $\cC(L)$ is a finitely generated $\coC[L]$-module and $\coC[L]$ is a commutative ring (see for instance \cite{AM}, Proposition 5.1), then there exist $a_i(\lambda)\in \coC [\lambda]$ such that 
\begin{equation*}
    h_A (L,A)=0 \mbox{ with } h_A (\lambda, \mu)=\mu^d + a_{d-1}(\lambda) \mu^{d-1}+\cdots + a_0(\lambda) .
\end{equation*}
The basis $\cB$ presented in \eqref{eq-beta} is not uniquely determined.  We will obtain in Theorem  \ref{thm-NormalB} a uniquely determined basis of $\cC(L)$ as a $\coC [L]$-module.

\section{Differential Resultants}\label{sec-DiffRes}

Let us consider ordinary  differential operators $P$ and $Q$ with coefficients in a differential domain $\bbD$, with derivation $\partial$. In order to study the common factors of $P$ and $Q$, we will  consider the fraction field $\bbK$ of $\bbD$, with the natural extension of the derivation, that we denote again by $\partial$. 

\medskip

The ring of differential operators $\bbK [\partial]$ is a (left and right) Euclidean domain, see \cite{BGV}, Corollary 4.35, and by \cite{BGV}, Corollary 4.36 it is a left and right principal ideal domain. In addition, by \cite{BGV}, Corollary 4.29, $\bbK[\partial]$ is a unique factorization domain, see \cite{BGV}, Definition 4.12.
Given differential operators $P,Q\in \bbK[\partial]$, we denote a greatest common right divisor 
of $P$ and $Q$ by $\gcrd (P,Q)$.  If $\gcrd (P,Q)\in\bbK$, we call $P$ and $Q$ {\sf right coprime}.

\medskip

Following \cite{McW}, we will review next how the existence of a non-trivial right factor is equivalent to the existence of a non-trivial order-bounded linear combination $C P + D Q = 0$.
This result justifies the definition of the differential resultant or Sylvester resultant of two differential operators given in \cite{Cha}. 

Let $\cM_{\ell}$ be the $\bbK$-vector space of differential operators in $\bbK [\partial]$ of order strictly less than $\ell$.
Assuming that $\ord(P)=n$ and $\ord(Q)=m$, let us consider the linear map
\begin{equation*}
    \s_0: \cM_{n} \oplus \cM_m\rightarrow \cM_{n+m}\mbox{ defined by } (C,D)\mapsto C P+D Q.
\end{equation*}
In the basis $\{\partial^{\ell},\ldots, \partial,1\}$ for each $\cM_{\ell}$, the matrix of $\s_0$ is the differential Sylvester matrix $\s_0(P,Q)$, a squared matrix of size $n+m$, with entries in $\bbD$, whose rows are the coefficient of the extended system
\begin{equation}\label{eq-E0}
    \Xi_0=\{\partial^{m-1} P,\ldots ,\partial P, P, \partial^{n-1} Q,\ldots ,\partial  Q, Q\}.
\end{equation}
The {\sf differential (Sylvester) resultant} of $P$ and $Q$ is the determinant 
\begin{equation}\label{eq-S0}
    \dres(P,Q):=\det (\s_0(P,Q)).
\end{equation}
Observe that the image of $\s_0$, denoted by $\Im (\s_0)$, is the $\bbK$-vector space spanned by $\Xi_0$. 

\begin{lem}\label{lem-linComb}
Given $P,Q$ in $\bbD [\partial]$ then
$\dres(P,Q)$ belongs to the image of $\s_0$, that is 
$$\dres(P,Q)=C_0 P+ D_0 Q\in \Im (\s_0)\cap \bbD,$$ 
with $\ord(C_0)=m-1$ and  $\ord(D_0)=n-1$.
\end{lem}
\begin{proof}
Multiply the Sylvester matrix on the right by the squared matrix  $E$ of size $n+m$ whose $i$-th row contains $1$ in the $i$-th position and $\partial^{n+m-i}$ in the last position. Observe that $\det E=1$ and that $\s_0(P,Q) E$ has entries in $\bbD [\partial]$, in particular its last column contains the elements of $\Xi_0$.
Thus
\[\det(\s_0(P,Q))=\det (\s_0(P,Q) E)=C_0 P+ D_0 Q\]
after developing $\s_0(P,Q) E$ by its last column.
\end{proof}

Let us consider the left ideal generated by $P$ and $Q$ in $\bbK [\partial]$ and denote it by $(P,Q)=\bbK [\partial] P+\bbK [\partial] Q$.
Observe that $\Im (\s_0)\subset (P,Q)$ and, by the previous lemma if $\dres (P,Q)\neq 0$ then the elimination ideal $(P,Q)\cap \bbD$ is nonzero since 
\begin{equation*}
  \dres(P,Q)\in (P,Q)\cap \bbD.  
\end{equation*}

As announced, we prove next that the existence of an order bounded linear combination $CP+DQ=0$, in other words $\Ker(\s_0)\neq 0$, is equivalent to $\gcrd(P,Q)\notin \bbK$.

\begin{thm}\label{lem-elimIdeal}
    Let us consider $P,Q\in \bbD[\partial]$. The following statements are equivalent:
    \begin{enumerate}
    \item $\dres (P,Q)\neq 0$.
        \item $\Im (\s_0)\cap \bbD\neq 0$.
        \item $P$ and $Q$ are right coprime in $\bbK [\partial]$.
    \end{enumerate}
\end{thm}
\begin{proof}
By Lemma \ref{lem-linComb}, 1 implies 2. 
To prove that 2 implies 3, let us assume that $\gcrd (P,Q)=F$ is a differential operator in $\bbK [\partial]$ of order greater than zero. Since $\bbK [\partial]$ is a left principal ideal domain then by \cite{BGV}, Proposition 4.16, (2) we have $\bbK [\partial] P+\bbK [\partial] Q=\bbK [\partial] F$. Given $e\in \Im(\s_0)\cap \bbD$  then $e\in \bbK [\partial] F$ and there exists $0\neq d\in\bbD$ such that $d e =L F^*$, with $L, F^*\in \bbD [\partial]$ and $F^*\notin \bbD$. Therefore $e=0$. 

If $\dres(P,Q)=\det (\s_0(P,Q))=0$ then $\s_0$ is not surjective so $\gcrd(P,Q)\notin \bbK$, proving that 3 implies 1.

\end{proof}

\subsection{First differential subresultant}\label{app-subresultant}

The differential resultant of two differential operators is a condition on their coefficients that guaranties the existence of a right common factor. We introduce next the first differential subresultant, as a tool to compute such factor in case it is a first order factor. Differential subresultants were introduced in \cite{Cha}, see also \cite{Li}.

Let us consider the $\bbK$-linear map
\begin{equation}
    \s_1: \cM_{n-1} \oplus \cM_{m-1}\rightarrow \cM_{n+m-1}\mbox{ defined by } (C,D)\mapsto C P+D Q.
\end{equation}
In the basis 
$\{\partial^{\ell},\ldots, \partial,1\}$ for each $\cM_{\ell}$, the matrix of $\s_1$ is a matrix $\s_1(P,Q)$ with $n+m-2$ rows and $n+m-1$ columns, with entries in $\bbD$, whose rows are the coefficient of the extended system
\begin{equation}\label{eq-E1}
    \Xi_1=\{\partial^{m-2} P,\ldots ,\partial P, P, \partial^{n-2} Q,\ldots ,\partial  Q, Q\}.
\end{equation}
Observe that  $\Im (\s_1)$ is the $\bbK$-vector space spanned by $\Xi_1$ and that $\Im (\s_1)\subset (P,Q)$.
Let us consider the squared matrix $\hat{\s}_1(P,Q)$ obtained by adding $\s_1(P,Q) $ the first row $(0, \cdots , 0, 1, -\partial)$.
The {\sf first differential subresultant} of $P$ and $Q$ is the determinant 
\begin{equation}\label{eq-Sij}
\sdres_1(P,Q):=\det \hat{\s}_1(P,Q)=\det(S_1^0)+\det(S_1^1)\partial , 
\end{equation}
where $S_1^0$ and $S_1^1$
are the submatrices of $\s_1 (P,Q)$ obtained by removing the columns indexed by $\partial$ and $1$ respectively.

\begin{lem}\label{lem-linCombS1}
Given $P,Q$ in $\bbD [\partial]$ then
$\dres(P,Q)$ belongs to the image of $\s_1$, that is 
$$\sdres_1(P,Q)=C_1 P+ D_1 Q\in \Im (\s_1),$$ 
with $\ord(C_1)=m-2$ and  $\ord(D_1)=n-2$.
\end{lem}
\begin{proof}
Multiply $\hat{\s}_1(P,Q)$ on the right by the squared matrix  $E$ of size $n+m-1$ whose $i$-th row contains $1$ in the $i$-th position and $\partial^{n+m-1-i}$ in the last position. Observe that $\det E=1$ and that $\hat{\s}_1(P,Q) E$ has entries in $\bbD [\partial]$, in particular its last column contains $0$ followed by the elements of $\Xi_1$.
Thus
\[\det(\s_1(P,Q))=\det (\hat{\s}_1(P,Q) E)=C_1 P+ D_1 Q\]
after developing $\s_1(P,Q) E$ by its last two columns.
\end{proof}

\begin{thm}[First Subresultant Theorem]\label{thm-FST}
    Let us consider $P,Q\in \bbD[\partial]$. If $\dres(P,Q)= 0$ and $\sdres_1(P,Q)\neq 0$ then
    \begin{enumerate}
        \item $\sdres_1(P,Q)$ is a differential operator of order $1$,
        \item $\gcrd(P,Q)$ equals $\sdres_1(P,Q)$ up to multiplication by an element of $\bbK$.
    \end{enumerate}
\end{thm}
\begin{proof}
    If $\sdres_1(P,Q)$ has order zero then it equals $\det(S_1^0)$. Therefore $\Im(\s_0)\cap \bbD\neq 0$ contradicting, by Theorem \ref{lem-elimIdeal}, that $\dres(P,Q)=0$.

    By Lemma \ref{lem-linCombS1}, we have $\sdres_1(P,Q)\in \Im (\s_1)\subset (P,Q)$ then $\sdres_1(P,Q)$ is a multiple of $\gcrd (P,Q)$, which proves 2. 
\end{proof}

\subsection{Characterizing common solutions}\label{sec-commonsol}

Given a differential operator $L\in \Sigma [\partial]$, a Picard–Vessiot extension $K$ of $\Sigma$ for $L$ is a differential field extension
$K = \Sigma \langle \psi_1,\ldots ,\psi_n \rangle$, where $\{\psi_1,\ldots ,\psi_n\}$ is a fundamental set of solutions, with no new
constants in $K$. It is the equivalent of a splitting field for $L(y) = 0$. Under the assumption that the field of constants $\coC$ is algebraically closed, the classical Picard-Vessiot theory guaranties the existence and uniqueness (up to differential automorphisms) of the Picard-Vessiot field for the equation $L(y) = 0$. See \cite{VPS}, Proposition 1.22.

Given differential operators $P$ and $Q$ in $\Sigma[\partial]$, let us consider the Picard-Vessiot extensions $(\cE_P,\partial_P)$ and $(\cE_Q,\partial_Q)$ of $\Sigma$ for the equations $P(y)=0$ and $Q(y)=0$ respectively, whose field of constants is $\coC$. 
As a consequence $\Sigma[\partial]$ can be included in $\cE_P[\partial_P]$ and $Q$ can be naturally extended to act on $\cE_P$. Analogously $P$ can be naturally extended to act on $\cE_Q$.

We give next a proof of Poisson's formula in \cite{Cha}, Theorem 5 using the strategy of E. Previato in \cite{Prev}. Given a fundamental system of solutions $\psi_1,\ldots ,\psi_n$ of $P(y)=0$ in $(\cE_P,\partial_P)$, we denote its  wronskian matrix by 
\begin{equation*}
    W(\psi_i)=\begin{pmatrix} 
    \psi_1 &\cdots &\psi_n\\
     \partial_P(\psi_1) &\cdots &\partial_P(\psi_n)\\
\vdots &\ddots &\vdots\\
      \partial^{n-1}_P(\psi_1) &\cdots &\partial^{n-1}_P(\psi_n)\\
    \end{pmatrix}.
\end{equation*}
{Note that the natural extension of $Q$ to $(\cE_P,\partial_P)$ allows to consider its action on solutions of $P(y)=0$ and to write $Q(\psi_i)$. Analogously, the action of $P$ on solutions of $Q(y)=0$ is naturally defined.}

\begin{thm}[Poisson's Formula]\label{thm-Poisson}
Let us consider differential operators $P$ and $Q$ in $\Sigma [\partial]$ of orders $n$ and $m$, and leading coefficients $a_n$ and $b_m$ respectively. Given fundamental systems of solutions $\psi_1,\ldots ,\psi_n$ of $P(y)=0$ in $\cE_P$ and $\phi_1,\ldots ,\phi_m$ of $Q(y)=0$ in $\cE_Q$ then
\begin{equation}\label{eq-Poisson}
    \dres(P,Q)=a_n^m \frac{\det W(Q(\psi_i))}{\det W(\psi_i)}
    =(-1)^{mn} b_m^n \frac{\det W(P(\phi_i))}{\det W(\phi_i)}.
\end{equation}
\end{thm}
\begin{proof}
Let us consider the following decomposition of the Sylvester matrix 
\begin{equation*}
    S_0(P,Q)=\left(\begin{array}{cc}
    M_1 & M_2\\
    M_3 & M_4 
    \end{array}\right)
\end{equation*}
where $M_1$ is an upper triangular $m\times m$ matrix with $a_n$ in the main diagonal. By the following matrix manipulation, where $I_n$ is the identity matrix of size $n$ and $\0$ is the zero matrix of the appropriate size
\begin{equation*}
    \left(\begin{array}{cc}
    M_1 & M_2\\
    M_3 & M_4 
    \end{array}\right) 
    \left(\begin{array}{cc}
    M_1^{-1} & -M_1^{-1} M_2\\
    \0 & I_n 
    \end{array}\right) = 
    \left(\begin{array}{cc}
    I_m & \0 \\
    M_3 M_1^{-1} & M_4-M_3 M_1^{-1} M_2
    \end{array}\right)
\end{equation*}
we obtain that 
\begin{equation}
    \dres(P,Q)=a_n^m\det(M_4-M_3 M_1^{-1} M_2).
\end{equation}
From the action of the system $\Xi_0$ on the fundamental system of solutions $\{\psi_i\}$ we obtain
\begin{equation}\label{eq-WQ}
    W(Q(\psi_i))=(M_4-M_3 M_1^{-1} M_2) W(\psi_i).
\end{equation}
Taking determinants in \eqref{eq-WQ}, the first part of formula \eqref{eq-Poisson} follows. 
\end{proof}

We are ready to review the Resultant Theorem. 

\begin{thm}[Resultant Theorem]\label{thm-DR}
Let us consider differential operators $P$ and $Q$ in $\Sigma [\partial]$. Let $\cE$ be a Picard-Vessiot extension of $\Sigma$ for $P(y)=0$ (or $Q(y)=0$).
Then the system 
\begin{equation}\label{eq-systemPQ}
    P(y)=0 \, , \, Q(y)=0
\end{equation}
has a nontrivial solution in $\cE$ if and only if $\dres(P,Q)=0$. 
\end{thm}
\begin{proof}
Let us consider a fundamental systems of solutions $\psi_1,\ldots ,\psi_n$ of $P(y)=0$ in $(\cE_P,\partial_P)$, whose field of constants is $\coC$. Thus $V=\oplus_i\coC \psi_i \subset \cE_P$ is the solution set of $P(y)=0$. A nontrivial solution of the system  \eqref{eq-systemPQ} in $\cE_P$ is therefore a nonzero $\psi\in V$ such that $Q(\psi)=0$.

By Poisson formula $\dres(P,Q)=0$ if and only if $\det(W(Q(\psi_i)))=0$. Equivalently, the columns of $W(Q(\psi_i))$ must be linearly dependent over $\coC$, namely for some $c_i\in \coC$ not all zero
\begin{equation*}
    \sum_i c_i \partial^j_P (Q(\psi_i))=0,\,\, j=0, 1,\ldots ,n-1.
\end{equation*}
In other words, $\det(W(Q(\psi_i)))=0$ is equivalent to the existence of a nonzero 
$\psi=\sum_i c_i \psi_i$ in  $V$ such that $Q(\psi)=0$.
\end{proof}

\begin{rem}
   {Observe that as a result of theorems \ref{lem-elimIdeal} and \ref{thm-DR} the differential operators $P$ and $Q$ have a common factor $F$ in $\Sigma [\partial]$ if and only if they have a non trivial solution in $\cE_P\cap \cE_Q$.}
\end{rem}

\bibliographystyle{plain}

\bibliography{Bibliography.bib}



\end{document}